\documentclass[a4paper,reqno]{amsart}
\usepackage{amsfonts}
\usepackage{amssymb}
\usepackage{enumitem}
\usepackage[colorlinks=true]{hyperref}
\hypersetup{urlcolor=blue,citecolor=blue}
\numberwithin{equation}{section}
\newtheorem{thm}{Theorem}[section]
\newtheorem{lem}[thm]{Lemma}

\newtheorem{prop}[thm]{Proposition}
\theoremstyle{definition}
\newtheorem{rem}[thm]{Remark}
\newcommand\R{{\mathbb R}}

\newcommand\N{{\mathbb N}}

\newcommand\cont{{\mathcal C}}
\newcommand\Ens{{\mathcal E}}
\newcommand\Fns{{\mathcal F}}
\newcommand\Cset{ K }

\newcommand\GCT[1]{\Lambda_{#1}}
\newcommand\GDT[1]{A_{#1}}

\newcommand\MScN[1]{\href{http://www.ams.org/mathscinet-getitem?mr=#1}{\nolinkurl{(#1)}}}
\newcommand\DOI[1]{\href{http://dx.doi.org/#1}{(doi: \nolinkurl{#1})}}
\newcommand\LINK[1]{\href{#1}{(link: \nolinkurl{#1})}}

\begin{document}
\title[Blow-up set for the nonlinear wave equation]
{Solutions blowing up on any given compact set for the energy subcritical wave equation}
\author[T. Cazenave]{Thierry Cazenave$^1$}
\email{\href{mailto:thierry.cazenave@sorbonne-universite.fr}{thierry.cazenave@sorbonne-universite.fr}}

\author[Y. Martel]{Yvan Martel$^{2}$}
\email{\href{mailto:yvan.martel@polytechnique.edu}{yvan.martel@polytechnique.edu}}

\author[L. Zhao]{Lifeng Zhao$^3$}
\email{\href{mailto:zhaolf@ustc.edu.cn}{zhaolf@ustc.edu.cn}}

\address{$^1$Sorbonne Universit\'e, CNRS, Universit\'e de Paris, Laboratoire Jacques-Louis Lions,
B.C. 187, 4 place Jussieu, 75252 Paris Cedex 05, France}

\address{$^2$CMLS, \'Ecole Polytechnique, CNRS, 91128 Palaiseau Cedex, France}

\address{$^3$Wu Wen-Tsun Key Laboratory of Mathematics and School of Mathematical Sciences, University of Science and Technology of China, Hefei 230026, Anhui, China}



\begin{abstract}
We consider the focusing energy subcritical nonlinear wave equation $\partial_{tt} u - \Delta u= |u|^{p-1} u$ in ${\mathbb R}^N$, $N\ge 1$.
Given any compact set $ K \subset {\mathbb R}^N $, we construct finite energy solutions which blow up at $t=0$ exactly on $ K $.

The construction is based on an appropriate ansatz. The initial ansatz is simply $U_0(t,x) = \kappa (t + A(x) )^{ -\frac {2} {p-1} }$, where $A\ge 0$ vanishes exactly on $ K $, which is a solution of the ODE $h'' = h^p$. 
We refine this first ansatz inductively using only ODE techniques and taking advantage of the fact that (for suitably chosen $A$), space derivatives are negligible with respect to time derivatives. 
We complete the proof by an energy argument and a compactness method.
\end{abstract}
\maketitle

\section{Introduction}
We consider the focusing nonlinear wave equation on $\R^N$
\begin{equation}\label{wave}
\partial_{tt} u - \Delta u= |u|^{p-1} u, \quad (t,x)\in\R\times\R^N,
\end{equation}
for any space dimension $N\geq 1$,
and   energy subcritical nonlinearities, \emph{i.e.}
\begin{equation}\label{on:p}
\mbox{$1<p<\infty$ if $N=1,2$ and } 1<p<\frac{N+2}{N-2} \mbox{ if $N\geq 3$}.
\end{equation}
It is well-known that under such condition on $p$ the Cauchy problem for \eqref{wave} is locally well-posed in the energy space $H^1(\R^N)\times L^2(\R^N)$
(see \cite{GV85, GV89, Segal}). For $H^1\times L^2$ solutions, the energy 
\[
E(u(t),\partial_tu(t))=
\int \left\{ \frac 12  |\partial_t u(t,x)|^2 +\frac 12  |\nabla u(t,x)|^2 - \frac 1{p+1}  |u(t,x)|^{p+1}\right\}dx
\]
is conserved through time. Moreover, it is known how to produce solutions blowing up in finite time
(see \emph{e.g.} \cite{Keller,Levine}).

Our main result states that for any given compact set $ \Cset $ of $\R^N$, there exists a finite-energy solution of \eqref{wave} which blows up in finite time exactly on~$ \Cset $.

\begin{thm}\label{TH:1}
Let $p$ satisfy \eqref{on:p} and let $ \Cset $ be any nonempty compact set of $\R^N$. There exist $\delta_0>0$ and a solution $(u,\partial_t u)\in \cont((0,\delta_0];H^1(\R^N)\times L^2(\R^N))$ of \eqref{wave}
 which blows up at time $0$ exactly on~$ \Cset $ in the following sense.
\begin{itemize}
\item If $x_0\in  \Cset $ then for any $r>0$,
\begin{equation}\label{xE}
\lim_{t\downarrow 0} \|u(t)\|_{L^2(|x-x_0|<r)}=\infty\quad \mbox{and}\quad \lim_{t\downarrow 0} \|\partial_t u(t)\|_{L^2(|x-x_0|<r)}=\infty.
\end{equation}
\item If $x_0\not\in  \Cset $ then there exists $r>0$ such that
\begin{equation}\label{xnotE}
\sup_{t\in(0,\delta_0]}\left\{\|u(t)\|_{L^2(|x-x_0|<r)}+\|\nabla u(t)\|_{L^2(|x-x_0|<r)}+\|\partial_t u(t)\|_{L^2(|x  -x_0 |<r)}\right\}<\infty.
\end{equation}
\end{itemize}
\end{thm}
\begin{rem}\label{rk:2}
For $t>0$, the function
\begin{equation}\label{def:h}
h(t)=\kappa t^{-\frac {2}{p-1}}\quad\mbox{where}\quad \kappa=\left[\frac {2(p+1)}{(p-1)^2}\right]^{\frac 1{p-1}}
\end{equation}
is a solution of the ordinary differential equation $h''=h^p$ which blows up at time~$0$.
It is also a solution of \eqref{wave}, but of course it fails to be in the energy space.
The function $h$ is the building block for our construction, it is thus relevant to compare it with the blow-up rate of the  solutions constructed in Theorem~\ref{TH:1}.
It follows from the proof that for any $0<\mu<\frac2{p-1}$ there exist solutions $u$ as in the statement of Theorem~\ref{TH:1}
satisfying in addition the following estimates: for any $x_0\in  \Cset $  $r>0$, and all $t\in (0,\delta_0]$,
\begin{gather}
 t^{-\mu} \lesssim \|u(t)\|_{L^2(|x-x_0|<r)}\lesssim t^{-\frac {2}{p-1}},\label{Qrp}\\
 t^{-\mu-1} \lesssim\|\partial_t u(t)\|_{L^2(|x-x_0|<r)}\lesssim t^{-\frac {2}{p-1}-1}.\label{dtQrp}
\end{gather}
Moreover, if $x_0\in  \Cset $ and $ \Cset $ contains a neighborhood of $x_0$ then it also holds, for any $r>0$,
and all $t\in (0,\delta_0]$,
\begin{equation}\label{jus}
\|u(t)\|_{L^2(|x-x_0|<r)}\gtrsim t^{-\frac {2}{p-1}}\quad \hbox{and}\quad
\|\partial_t u(t)\|_{L^2(|x-x_0|<r)}\gtrsim t^{-\frac {2}{p-1}-1}.
\end{equation}
In contrast, if $x_0$ is an isolated point of the compact set $ \Cset $, solutions $u$ as in Theorem~\ref{TH:1} can be chosen so that, for a small $r>0$,
\[
\lim_{t\downarrow 0} \left\{t^{\frac {2}{p-1}} \|u(t)\|_{L^2(|x-x_0|<r)}+
t^{\frac {2}{p-1}+1}\|\partial_t u(t)\|_{L^2(|x-x_0|<r)}\right\}=0.
\]
\end{rem}
To prove Theorem~\ref{TH:1}, we follow the strategy developed in \cite{CaMaZh} to construct blow-up solutions of ODE type for a class of semilinear Schr\"odinger equations. First, we construct an approximate solution to the blow-up problem based on the explicit blow-up solution $h$ defined by \eqref{def:h}. The main order term of the approximate solution is $U_0(t,x)=\kappa(t+A(x))^{-\frac{2}{p-1}}$, where $A$ is a suitable nonnegative function which vanishes exactly on $ \Cset $ and whose behavior at $\infty$ ensures that $U_0$ belongs to the energy space. Typically, to obtain blowup at only one point $x_0$, it suffices to consider $A(x)=|x-x_0|^k$ for $k$ large enough.
Compared to \cite{CaMaZh} where a simple ansatz such as $U_0$ is sufficient, at least for strong enough nonlinearities, the wave equation requires to introduce iterated refinements $U_J$ of this ansatz (the number of iterations $J\geq 1$ depends on $p$, see Remark~\ref{eRemStr1}). The basic idea is that for such blow-up profiles, the space derivatives are of lower order compared to time derivatives and to nonlinear terms. This allows to use only elementary arguments of ordinary differential equations for the construction of the refined ansatz $U_J(t,x)$, at fixed $x$. See Section~\ref{sec:2}.
(The construction by purely ODE techniques of an approximate blow-up solution to an Euler-Poisson system is done in~\cite{GuoHJ}.)

Second, we consider the sequence $(u_n)$ of solutions of the wave equation \eqref{wave} with initial data $u_n(\frac 1n)=U_J(\frac 1n)$. Using energy method in $H^1\times L^2$, we prove uniform estimates on this sequence on intervals
$[\frac 1n,\delta_0]$, where $\delta_0>0$ is uniform in~$n$ (see Section~\ref{sec:3}).
Passing to the limit $n\to \infty$ yields the solution $u$ of Theorem~\ref{TH:1}.

We point out that this strategy by approximate solution and compactness is also reminiscent to \cite{Martel,Merle1,RaphaelSzeftel} where global or blow-up solutions with special asymptotic behavior  are constructed using the reversibility of the equation and suitable uniform estimates on backwards solutions. 

For stability results concerning the solution $h$ \eqref{rk:2}, we refer to~\cite{Don-Sch}.  
For ODE-type blowup for quasilinear wave equations, see~\cite{Speck} and the references therein. 
We also refer to \cite{CollotGM}
where an ODE   blow-up profile similar to $U_0$ is used to construct blow-up solutions of the nonlinear heat equation with applications to the Burgers equation.

In this article, we restrict ourselves to energy subcritical power nonlinearities for simplicity, since
this framework allows us to use the energy method at the level of regularity $H^1\times L^2$ only.
However, the approximate solutions constructed in Section~\ref{sec:2} are relevant for any power nonlinearity, 
and we expect that a higher order energy method (to estimate higher order Sobolev norms) should be sufficient to extend the construction to energy critical or supercritical nonlinearities (at least for integer powers to avoid regularity issues).
\begin{rem} 
A more general question for nonlinear wave equations concerns the blow-up surface.
For a solution of \eqref{wave} with initial data at $t=0$, which is assumed to blow up in finite time, there exists a $1$-Lipschitz function $x\mapsto \phi(x)>0$ such that the solution is well-defined in a suitable sense in the maximal domain of influence
$D=\{(t,x):0\leq t<\phi(x)\}$, see \emph{e.g.} \cite{Alinhac}, Sections~III.2 and~III.3.
The surface $\{(\phi(x),x):x\in\R^N\}$ is called the blow-up surface.
The question of the regularity of blow-up surface is adressed in \cite{Alinhac,CaFr1, CaFr2,MerleZ2,MerleZ3}.
The question of constructing solutions of the nonlinear wave equation with prescribed blow-up surface (with sufficient regularity
and satisfying the space-like condition $\|\nabla \phi\|_{L^\infty}<1$) is also a classical question, adressed in several articles and books, notably  \cite{KiLi1, KiLi2}, \cite{Kiche1,Kiche2,Kibook}, \cite{KiQi} 
and \cite{Alinhac}. 
The approach by Fuschian reduction is especially well-described in the book \cite{Kibook}. 
First developed for analytic surfaces and exponential nonlinearity, this method was later extended to  surfaces with Sobolev regularity and to some power nonlinearities.
However, it is not clear to us whether the strategy described in \cite{Kibook} for constructing solutions with given blow-up surface can be extended to power nonlinearities $|u|^{p-1}u$ for any $p>1$, or to more general nonlinearities.

As discussed in \cite{Kibook,Kicontrol,KiQi}, prescribing the blow-up set of a blow-up solution can be seen as a sub-product of prescribing its blow-up surface. The solutions constructed in \cite{Kibook,Kicontrol,KiQi} may only exist in a space-time region around the blow-up surface, which does not guarantee that the solution is globally defined in space at any one specific time.
However, in the one dimensional case \cite[Corollary 1.2]{KiQi} actually proves the existence of smooth initial data leading to blowup on arbitrary compact set of $\R$, for any power nonlinearity.

We also would like to point out a difference between the above mentioned articles and our approach. Here, we resolutely work with finite energy solutions and the initial value problem for~\eqref{wave}. It is often argued that  finite speed of propagation and cut-off arguments allow to reduce to finite energy solutions. 
For example, the function \eqref{def:h} is used to claim that ODE-type blowup   is easy to reach for finite energy solutions. However, the cut-off necessary to localize the initial data could lead to blowup in an earlier time. Our method deals with these issues by constructing directly a finite energy solution with initial data
from a finite energy ansatz.
Moreover, we hope  that our somehow elementary approach can be of interest for its simplicity and its large range of applicability to other more complicated problems where ODE blowup is relevant.
\end{rem}
\subsection*{Notation}
We fix a smooth, even function $\chi:\R\to\R$ satisfying:
\begin{equation}\label{def:chi}
\hbox{$\chi\equiv 1$ on $[0,1]$, $\chi\equiv 0$ on $[2,\infty)$ and $\chi'\leq 0\leq \chi\leq 1$ on $[0,\infty)$.}
\end{equation}
For $p>1$ satisfying \eqref{on:p}, recall the well-known inequality, for any $u\in H^1$,
\begin{equation}\label{gagl}
\|u\|_{L^{p+1}}^{p+1} \lesssim\|u\|_{L^2}^{p+1-\frac{N}2(p-1)}\|\nabla u\|_{L^2}^{\frac{N}2(p-1)}.
\end{equation}
Let $f(u)=|u|^{p-1} u$ and $F(u)=\int_0^u f(v) dv$.
For future reference, we recall Taylor's formulas involving the functions $F$ and $f$.
Let $\bar p = \min(2,p)$.
First, we claim that for any $u>0$ and $v\in \R$,
\begin{equation}\label{taylor0}
\Big|F(u+v)-F(v)-F'(u)v-\frac 12F''(u)v^2\Big|
\lesssim |v|^{p+1}+u^{p-\bar p} |v|^{\bar p+1}.
\end{equation}
Indeed, in the region $|v|\geq \frac 12 u$, each term on the left-hand side is
bounded by $|v|^{p+1}$.
In the region $|v|\leq \frac 12 u$, we use Taylor's expansion to write
\begin{equation*}
\Big|F(u+v)-F(u)-F'(u)v-\frac 12F''(u)v^2\Big|\lesssim u^{p-2} |v|^3.
\end{equation*}
If $p\geq 2$, then $\bar p=2$ and \eqref{taylor0} is proved. If $1<p<2$, we finish by saying that in this case $u^{p-2} |v|^3\lesssim |v|^{p+1}$.
The same argument shows that
\begin{equation}\label{taylor1}
| ( f(u+v)- f( u )- f'(u)v) v |
\lesssim |v|^{p+1}+u^{p-\bar p} |v|^{\bar p+1}.
\end{equation}
Next, we claim that for any $u>0$ and $v\in \R$,
\begin{equation}\label{taylor}
\Big|f(u+v)-f(u)-f'(u)v-\frac 12f''(u)v^2\Big|\lesssim u^{-1}|v|^{p+1}+u^{p-\bar p-1}|v|^{\bar p+1}.
\end{equation}
Indeed, in the region $|v|\geq \frac 12 u$, each term on the left-hand side is
bounded by $ u^{-1} |v|^{p+1} $, and~\eqref{taylor} follows.
In the region $|v|\leq \frac 12 u$, we use Taylor's expansion to write
\begin{equation*}
\Big|f(u+v)-f(u)-f'(u)v-\frac 12f''(u)v^2\big|\lesssim u^{p-3} |v|^3.
\end{equation*}
If $p\geq 2$, then $\bar p=2$ and \eqref{taylor} is proved. If $1<p<2$, we finish by saying that in this case $u^{p-3} |v|^3\lesssim u^{-1} |v|^{p+1}$.

In this article, we will use multi-variate notation and results from \cite{CoSa}.
For any $\beta=(\beta_1,\ldots,\beta_N)\in \N^N$, $x=(x_1,\ldots,x_N)\in \R^N$, we set
\begin{align*}
&|\beta|=\sum_{j=1}^N \beta_j,\quad \beta!=\prod_{j=1}^N (\beta_j!),
\quad x^\beta=\prod_{j=1}^N x_j^{\beta_j},\\
& \partial_x^0=\mathop{\textnormal{Id}}\quad \mbox{and}\quad 
\partial_x^{\beta} = \frac{\partial^{|\beta|}}{\partial_{x_1}^{\beta_1}\ldots\partial_{x_N}^{\beta_N}}
\quad\mbox{for $|\beta|>0$.}
\end{align*}
For $\beta,\beta'\in \N^N$, we write
$\beta'\leq \beta$ if $\beta_j'\leq \beta_j$ for all $j=1,\ldots,N$. 
When $\beta'\leq \beta$, we set
\[
\binom{\beta}{\beta'}=\prod_{j=1}^N\binom{\beta_j}{\beta_j'} = \frac{\beta!}{\beta'!(\beta-\beta')!}.
\]
With this notation, given two functions $a,b:\R^{N}\to \R$, Leibniz's formula writes:
\begin{equation}\label{lbz0}
\partial_x^\beta \left(ab\right)=
\sum_{\beta'\leq\beta}\binom{\beta}{\beta'}\left(\partial_x^{\beta'}a\right)\left(\partial_x^{\beta-\beta'}b\right).
\end{equation}
We write $\beta'\prec\beta$ if one of the following holds
\begin{itemize}
\item $|\beta'|<|\beta|$;
\item $|\beta'|=|\beta|$ and $\beta_1'<\beta_1$;
\item $|\beta'|=|\beta|$, $\beta_1'=\beta_1$,\ldots, $\beta_\ell '=\beta_\ell $ and
$\beta_{\ell +1}'<\beta_{\ell +1}$ for some $1\leq \ell <N$.
\end{itemize}
Finally, we recall the Faa di Bruno formula (see Corollary~2.10 in~\cite{CoSa}).
Let $n=|\beta|\geq 1$. Then, for functions $q:\R\to \R$, $a:\R^{N}\to \R$,
\begin{equation}\label{fdb0}
\partial_x^{\beta} (q\circ a)=
\sum_{\nu=1}^n \left(q^{(\nu)}\circ a\right)\sum_{P(\beta,\nu)}(\beta!)
\prod_{ \ell =1}^n \frac{\left(\partial_x^{\beta_\ell }a\right)^{\nu_\ell }}{(\nu_\ell !)(\beta_\ell !)^{\nu_\ell }}
\end{equation}
where
\begin{align*}
P(\beta,\nu)
&=\Big\{(\nu_1,\ldots,\nu_n;\beta_1,\ldots,\beta_n): \mbox{there exists $1\leq m\leq n$ such that}\\
&\qquad \nu_\ell =0 \mbox{ and } \beta_\ell =0 \mbox{ for $1\leq \ell \leq n-m$} ;\, \nu_\ell >0 \mbox{ for $n-m+1\leq \ell \leq n$};\\
&\qquad \mbox{and } 0\prec\beta_{n-m+1}\prec\cdots\prec\beta_n \mbox{ are such that }
\sum_{\ell =1}^n \nu_\ell =\nu,\ \sum_{\ell =1}^n \nu_\ell  \beta_\ell  = \beta \Big\}.
\end{align*}
\section{The blow-up ansatz}\label{sec:2}
\subsection{Preliminary}\label{sec:2.1}
Recall that $h$ is the explicit solution \eqref{def:h} of the equation $h''=h^p$ which blows up at $0$.
The linearization of this equation around the solution $h$ yields the linear equation
\[
g''=ph^{p-1} g=\frac {2p(p+1)}{(p-1)^2}t^{-2} g
\]
which admits the following two independent solutions
\[
g_1(t)=t^{-\frac {p+1}{p-1}},\quad g_2(t)=t^{\frac {2p}{p-1}}, \quad  \text{for } t>0.
\]
Since $\frac{p+1}{p-1}>\frac 2{p-1}$, the function $g_1$, related to time invariance, is more singular at $0$ than the function $h$.
Note also that for a function $G$ satisfying $\int_0^1 s^{\frac{2p}{p-1}}| G(s)| ds<\infty$, a solution of the following linearized equation with source $G$
\[
g''=\frac {2p(p+1)}{(p-1)^2}t^{-2} g+G,
\]
is given by
\[
g(t)=-\frac{p-1}{3p+1}\left( t^{-\frac {p+1}{p-1}}\int_0^t s^{\frac {2p}{p-1}} G(s) ds
+t^\frac {2p}{p-1} \int_t^1 s^{-\frac {p+1}{p-1}} G(s) ds\right).
\]

\subsection{First blow-up ansatz}\label{s22}
Set
\begin{equation}\label{def:Jk}
J=\left\lfloor\frac {p+1}{p-1}\right\rfloor \quad \mbox{and} \quad k\geq 2J+2
\end{equation}
where $x\mapsto \lfloor x \rfloor$ is the floor function which maps $x$ to the greatest integer less than or equal to $x$.
(See Remark~\ref{eRemStr1} below for the explanation of the numbers  $J$ and $k$.)
We consider a function $A:\R^N\to\R$ of class $\cont^{k-1}$ on $\R^N$ and of class $\cont^k$ piecewise on $\R^N$ such that, for any $\beta\in \N^N$, with 
$|\beta|\leq k-1$, the following hold
\begin{equation} \label{on:A}
\begin{cases} 
A\geq 0 \hbox{ and }
|\partial_x^\beta A|\lesssim A^{1-\frac {|\beta|}{k}} &  \text{on }\R^N ,\\
 A(x)=|x|^k &  \text{for }x\in \R^N ,  |x|\ge 2 . 
\end{cases} 
\end{equation} 

\begin{rem}\label{rk:4}
Typical examples of such functions are $A(x):=|x|^k$, which vanishes at $0$ and 
\[
A(x):=
\begin{cases} 
0& \hbox{if $|x|\leq 1$}\\
(|x|-\chi(   |x| ))^k&\hbox{if $1< |x|\leq 2$}\\
|x|^k&\hbox{if $|x|>2$}
\end{cases}
\]
(where $\chi $ is given by~\eqref{def:chi}) which vanishes on the closed ball of center $0$ and radius $1$.
Another example, important for the proof of Theorem \ref{TH:1} is given in Section~\ref{S:proof}:
for any compact set $ \Cset $ of $\R^N$ included in the open ball of center $0$ and radius $1$,
there exists a function $A$ satisfying \eqref{on:A} which vanishes exactly on $ \Cset $.
\end{rem}
For $t>0$ and $x\in \R^N$, set
\[
U_0(t,x)=\kappa(t+A(x))^{-\frac{2}{p-1}}= h(W(t,x))\quad\mbox{where}\quad W(t,x)=t+A(x),
\]
so that $U_0$ satisfies $\partial_{tt} U_0 = f(U_0)$ on $(0,\infty)\times\R^N$.
Let 
\[
\Ens_0=- \partial_{tt} U_0 + \Delta U_0 -f(U_0)= \Delta U_0.
\]
We gather in the next lemma some estimates for $U_0$ and $\Ens_0$.
\begin{lem}\label{le:0}
The function $U_0$ satisfies
\begin{equation}\label{e:19}
\partial_t U_0 = -\left( \frac 2{p+1} U_0^{p+1}\right)^{\frac 12},\quad 
(\partial_t U_0)^2= \frac 2{p+1} U_0^{p+1}, \quad\partial_{tt} U_0 = U_0^{p}.
\end{equation}
Moreover, for any $\beta\in \N^N$, $\rho\in \R$, $0<t\leq 1$, $x\in\R^N$, the following hold.
\begin{enumerate}[label=\emph{(\roman*)}]
\item If $0\leq|\beta|\leq k-1$ and $|x|\leq 2$,
\begin{equation}\label{e:20}
|\partial_x^\beta(U_0^\rho)|\lesssim U_0^{\rho+\frac{|\beta|}{k}\frac{p-1}2},\quad
|\partial_t\partial_x^\beta(U_0^\rho)|\lesssim U_0^{\rho+(1+\frac{|\beta|}{k})\frac{p-1}2}.
\end{equation}
\item If $0\leq|\beta|\leq k-3$ and $|x|\leq 2$,
\begin{equation}\label{e:21}
|\partial_x^\beta\Ens_0|\lesssim U_0^{1+\frac{2+|\beta|}{k}\frac{p-1}2}.
\end{equation}
\item If $|x|>2$,
\begin{equation}\label{e:22}
|\partial_x^\beta U_0|\lesssim |x|^{-\frac{2k}{p-1}-|\beta|},
\quad
|\partial_x^\beta\Ens_0|\lesssim |x|^{-\frac{2k}{p-1}-|\beta|-2}.
\end{equation}
\end{enumerate}
Furthermore, for any $x_0\in \R^N$ such that $A(x_0)=0$, for any $r>0$, $0<t\leq 1$,
\begin{equation}\label{e:24}
t^{-\frac2{p-1}+\frac N{2k}}\lesssim \|U_0(t)\|_{L^2(|x-x_0|<r)}\lesssim t^{-\frac2{p-1}},
\end{equation}
\begin{equation}\label{e:25}
t^{-\frac2{p-1}-1+\frac N{2k}}\lesssim \|\partial_t U_0(t)\|_{L^2(|x-x_0|<r)}\lesssim t^{-\frac2{p-1}-1},
\end{equation}
where the implicit constants in \eqref{e:24} and \eqref{e:25} depend on $r$.
\end{lem}
\begin{proof}
The identities in~\eqref{e:19} follow from the definition of $U_0$ and direct calculations.

Proof of \eqref{e:20}-\eqref{e:21}. For $0<t\leq 1$ and $|x|\leq 2$, one has $0<t+A(x)\lesssim 1$ and thus
$U_0\gtrsim 1$. From $U_0=h\circ W$, setting $n=|\beta|$ and using \eqref{fdb0}, one has
\[
\partial_x^\beta U_0=
\sum_{\nu=1}^n \left(h^{(\nu)}\circ W\right)\sum_{P(\beta,\nu)}(\beta!)
\prod_{ \ell =1}^n \frac{\left(\partial_x^{\beta_\ell }W\right)^{\nu_\ell }}{(\nu_\ell !)(\beta_\ell !)^{\nu_\ell }}.
\]
For $\nu\geq 1$, we have $|h^{(\nu)}\circ W|\lesssim W^{-\frac 2{p-1}-\nu}$.
Moreover, using the assumption \eqref{on:A}, we have, for $1\leq |\beta_\ell |\leq k-1$,
\[
|\partial_x^{\beta_\ell }W|\lesssim |\partial_x^{\beta_\ell } A|\lesssim A^{1-\frac{|\beta_\ell |}{k}}.
\]
Since $\sum_{\ell =1}^n \nu_\ell =\nu$, $\sum_{\ell =1}^n \nu_\ell |\beta_\ell |=|\beta|$ and $|\beta|\leq k-1$, we obtain
\begin{align*}
|\partial_x^\beta U_0|&\lesssim 
\sum_{\nu=1}^n W^{-\frac 2{p-1}-\nu} \sum_{P(\beta,\nu)} \prod_{ \ell =1}^n \left( A^{1-\frac{|\beta_\ell |}{k}}\right)^{\nu_\ell }\\
& \lesssim \sum_{\nu=1}^n W^{-\frac 2{p-1}-\nu} A^{\nu-\frac{|\beta|}{k}}
\lesssim W^{-\frac2{p-1}-\frac{|\beta|}k}\lesssim U_0^{1+\frac{|\beta|}{k}\frac{p-1}2},
\end{align*}
which proves the first estimate of \eqref{e:20} for $\rho=1$.
For $\rho\in\R$, using \eqref{fdb0}, we also have, for $1\leq n=|\beta|\leq k-1$,
\[
\partial_x^\beta(U_0^\rho)
=\sum_{\nu=1}^n \left[\rho\cdots(\rho-\nu+1)\right]U_0^{\rho-\nu}\sum_{P(\beta,\nu)}(\beta!)\prod_{\ell =1}^n
\frac{(\partial_x^{\beta_\ell } U_0)^{\nu_\ell }}{(\nu_\ell !)(\beta_\ell !)^{\nu_\ell }}.
\]
Using the above estimate on $|\partial_x^\beta U_0|$ and $\sum_{\ell =1}^n \nu_\ell =\nu$, $\sum_{\ell =1}^n \nu_\ell  |\beta_\ell |=|\beta|$, 
we obtain
\begin{equation*}
|\partial_x^\beta (U_0^\rho)|
 \lesssim \sum_{\nu=1}^n U_0^{\rho-\nu} \sum_{P(\beta,\nu)}\prod_{\ell =1}^n U_0^{\nu_\ell  \left[1+\frac{|\beta_\ell |}{k} \frac{p-1}2\right]} \lesssim \sum_{\nu=1}^n U_0^{\rho-\nu}U_0^{\nu+\frac{|\beta|}{k}\frac{p-1}2}
\lesssim U_0^{\rho+\frac{|\beta|}{k}\frac{p-1}2}.
\end{equation*}
Next, using the first identity in~\eqref{e:19}, we see that $\partial _t U_0^\rho = -\rho ( \frac {2} {p+1} )^{\frac {1} {2}} U_0^{\rho + \frac {p-1} {2}}$; and so the second estimate in~\eqref{e:20} follows from the first. 
Since $\Ens_0=  \Delta U_0$, \eqref{e:21} is an immediate consequence of the first estimate in~\eqref{e:20}. 

Estimate~\eqref{e:22}
 is a direct consequence of the definitions of $U_0$ and $\Ens_0$ and of the fact that $A(x)=|x|^k$ for $|x|>2$.
 
Proof of \eqref{e:24}-\eqref{e:25}. For any $x_0\in \R^N$ and $r>0$, the upper bounds in \eqref{e:24} and \eqref{e:25} are direct consequences of the estimates $0\leq U_0\lesssim t^{-\frac 2{p-1}}$
and $|\partial_t U_0|\lesssim t^{-\frac 2{p-1}-1}$.
Let $x_0\in \R^N$ be such that $A(x_0)=0$ and $r>0$. By \eqref{on:A} and the fact that the function $A$ is of class $\cont^k$ piecewise,
the Taylor formula implies that for any $x$ such that $|x-x_0|<r$, $|A(x)|\leq C(r) |x-x_0|^k$.
It follows that for such $x$, and for any $t\in (0,1]$,
$U_0^2(t,x)=\kappa^2 (t+A(x))^{-\frac{4}{p-1}}\gtrsim (t+|x-x_0|^k)^{-\frac{4}{p-1}}$.
The lower estimate in~\eqref{e:24} then follows from
\begin{equation} \label{label4}
\begin{split} 
\int_{|y |<r} (t+| y |^k)^{-\frac{4}{p-1}} dy&=
t^{-\frac4{p-1}+\frac Nk}\int_{|z|<t^{-\frac1k}r} (1+|z|^k)^{-\frac{4}{p-1}} dz\\
&\gtrsim
t^{-\frac4{p-1}+\frac Nk}\int_{|z|< r} (1+|z|^k)^{-\frac{4}{p-1}} dz
\gtrsim t^{-\frac4{p-1}+\frac Nk}.
\end{split} 
\end{equation} 
Estimate \eqref{e:25} is proved similarly.
\end{proof}

\subsection{Refined blow-up ansatz}\label{s23}
Starting from $U_0$, we define by induction a refined ansatz to the nonlinear wave equation.
Let $t_0=1$ and for any $j\in \{1,\ldots,J\}$, let $0<a_j\leq 1$ and $0<t_j\leq 1$ to be chosen later.
Let
\begin{align*}
w_j& = - \kappa^{\frac{p-1}2} \frac{p-1}{3p+1}
\left(U_0^{\frac {p+1}2} \int_0^t U_0^{-p} \Ens_{j-1} ds
+ U_0^{-p} \int_{t}^{t_{j-1}} U_0^{\frac {p+1}2} \Ens_{j-1} ds\right),
\\
U_j& =U_0+ \sum_{\ell =1}^j \chi_\ell  w_\ell ,\quad 
\Ens_j=-\partial_{tt}U_j+\Delta U_j+f(U_j),
\end{align*}
where $\chi_j(x)=\chi( A(x)/{a_j})$ and  $\chi $ satisfies~\eqref{def:chi}.

\begin{lem}\label{le:1}
There exist $0<a_J\leq \cdots\leq a_1\leq 1$ and $0<t_J\leq \cdots\leq t_1\leq 1$
such that for any $0\leq j\leq J$, for any $\beta\in \N^N$, $0<t\leq t_j$ and $x\in \R^N$, the following hold.
\begin{enumerate}[label=\emph{(\roman*)}]

\item If $1\leq j\leq J$, $0\leq |\beta|\leq k-1-2j$, $|x|\leq 2$, then
\begin{equation}\label{e:29}
|\partial_x^\beta w_j|\lesssim U_0^{1-j(p-1)+\frac{2j +|\beta|}k \frac{p-1}2},
\end{equation}
\begin{equation}\label{e:30}
|\partial_t\partial_x^\beta w_j|\lesssim U_0^{\frac{p+1}2-j(p-1)+\frac{2j +|\beta|}k \frac{p-1}2}.
\end{equation}

\item If $1\leq j\leq J$, then
\begin{equation}\label{e:31}
|U_j-U_0| \leq \frac 14(1-2^{-j})U_0,\quad
|U_j-U_0|\leq (1-2^{-j})(1+U_0)^{-\frac{p-1}2} U_0,
\end{equation}
\begin{equation}\label{e:31bis}
|\partial_tU_j-\partial_t U_0|\lesssim U_0.
\end{equation}

\item If $0\leq |\beta|\leq k-3-2j$, $|x|\leq 2$, then
\begin{equation}\label{e:32}
|\partial_x^\beta \Ens_j|\lesssim U_0^{1-j(p-1)+\frac{2j+2 +|\beta|}k \frac{p-1}2}.
\end{equation}

\item If $|x|\geq 2$, then
\begin{equation}\label{e:33}
|\partial_x^\beta U_j |\lesssim |x|^{-\frac {2k}{p-1}-|\beta|},
\quad
|\partial_x^\beta \Ens_j |\lesssim |x|^{-\frac {2k}{p-1}-2-|\beta|}.
\end{equation}
\end{enumerate}
\end{lem}

\begin{rem} \label{eRemStr1}
We comment on the mechanism of the refined ansatz.
For the energy control which we establish in the next section, we need an estimate on the error term $ \|\mathcal E_J\| _{ L^2 } \lesssim t^{ (\frac {2} {p-1})^+ }$. (See formulas~\eqref{dtener}  and~\eqref{fForarap}.) By formula~\eqref{e:32}, this is achieved if $J > \frac {2} {p-1}$, which is the first condition in~\eqref{def:Jk}, and then $k$ sufficiently large (once $J$ is chosen), which is the second condition in~\eqref{def:Jk}. 
Note that for $p>3$, $J=1$ is enough, but one can never choose $J=0$, so a refined ansatz is always needed.
We see on formula~\eqref{e:32} that at each step, the error estimate improves by a factor $U_0^{ - (p-1) (1 - \frac {1} {k}) } \sim t^{2 ( 1 - \frac {1} {k})}$. It is clear then that the number of steps goes to $\infty $ as $p\to 1$. 
\end{rem} 

\begin{proof}[Proof of Lemma~$\ref{le:1}$]
Observe that \eqref{e:32} for $j=0$ is exactly \eqref{e:21} in Lemma~\ref{le:0}. 
Now, we proceed by induction on $j$: for any $1\leq j\leq J$, we prove that
 estimate \eqref{e:32} for $\Ens_{j-1}$ implies estimates \eqref{e:29}--\eqref{e:32} for $w_j$, $U_j$ and $\Ens_j$, for an appropriate choice of $a_j$ and~$t_{j}$.

\smallskip
Proof of \eqref{e:29}-\eqref{e:30}.
First, assuming \eqref{e:32} for $\Ens_{j-1}$, we show the following estimates related to the two components of $w_j$: for $|\beta|\leq k-1-2j$, $0<t\leq t_{j-1}$ and $|x|\leq 2$,
\begin{align}
\left|\partial_x^\beta\left(\int_0^t U_0^{-p} \mathcal E_{j-1} ds\right)\right| &\lesssim 
U_0^{-\frac{p-1}2-j (p-1)+\frac{2j+|\beta|}{k}\frac{p-1}2},\label{e:35}\\
\left|\partial_x^\beta\left(\int_{t}^{t_{j-1}} U_0^{\frac{p+1}2} \mathcal E_{j-1} ds\right)\right| &\lesssim 
U_0^{p+1-j (p-1)+\frac{2j+|\beta|}{k}\frac{p-1}2}.\label{e:36}
\end{align}
Indeed, we have by Leibniz's formula \eqref{lbz0}
\begin{equation*}
\partial_x^\beta \left(U_0^{-p}\Ens_{j-1}\right)=
\sum_{\beta'\leq\beta}\binom{\beta}{\beta'}\left(\partial_x^{\beta'}(U_0^{-p})\right)
\left(\partial_x^{\beta-\beta'}\Ens_{j-1}\right),
\end{equation*}
and thus, using \eqref{e:20} and \eqref{e:32},
\begin{align*}
|\partial_x^\beta \left(U_0^{-p}\Ens_{j-1}\right)|
&\lesssim 
\sum_{\beta'\leq\beta}\binom{\beta}{\beta'}
U_0^{-p+\frac{|\beta'|}k\frac{p-1}2}
 U_0^{1-(j-1)(p-1)+\frac{2j +|\beta|-|\beta'|}k \frac{p-1}2}\\
 &\lesssim U_0^{-j(p-1)+\frac{2j +|\beta|}k \frac{p-1}2}
 \lesssim (t+A)^\gamma,
\end{align*}
where for $j\geq 1$, $|\beta|\leq k$, 
\[
\gamma:=2j-\frac{2j+|\beta|}{k}=2j\left(1-\frac1k\right)-\frac{|\beta|}{k}\geq 0.
\]
Integrating on $(0,t)$ for $t\in (0,t_{j-1}]$, we obtain 
\[
\left|\partial_x^\beta\left(\int_0^t U_0^{-p} \mathcal E_{j-1} ds\right)\right| \lesssim
(t+A)^{\gamma+1}\leq U_0^{-\frac{p-1}2-j\left(1-\frac1k\right)(p-1)+\frac{|\beta|}{k}\frac{p-1}2},
\]
which is \eqref{e:35}.
Similarly, using Leibniz's formula, we check the following estimate
\[
|\partial_x^\beta ( U_0^{\frac{p+1}2}\Ens_{j-1})|
\lesssim U_0^{\frac{3p+1}2-j(p-1)+\frac{2j +|\beta|}k \frac{p-1}2}
 \lesssim (t+A)^{-\gamma'},
\]
where, using $0<j\leq J\leq \frac{p+1}{p-1}$,
\[
\gamma':=\frac{3p+1}{p-1}-2j+\frac{2j+|\beta|}{k}> 1.
\]
Thus, by time integration, for $t\in (0,t_{j-1}]$,
\[ 
\left| \int_t^{t_{j-1}} \partial_x^\beta ( U_0^{\frac{p+1}2}\Ens_{j-1}) ds\right|
\lesssim (t+A)^{-\gamma'+1}
\lesssim 
U_0^{p+1-j (p-1)+\frac{2j+|\beta|}{k}\frac{p-1}2},
\]
which is \eqref{e:36}.

Using Leibniz's formula, \eqref{e:20}, and~\eqref{e:35}-\eqref{e:36}, we deduce easily that, for any $\beta\in \N^N$, $|\beta|\leq k -1 -2j$,
\begin{align*}
\left|\partial_x^\beta\left( U_0^{\frac{p+1}2}\int_0^t U_0^{-p} \mathcal E_{j-1} ds\right)\right| &\lesssim 
U_0^{1-j (p-1)+\frac{2j+|\beta|}{k}\frac{p-1}2},\\
\left|\partial_x^\beta\left(U_0^{-p}\int_{t}^{t_{j-1}} U_0^{\frac{p+1}2} \mathcal E_{j-1} ds\right)\right| &\lesssim 
U_0^{1-j (p-1)+\frac{2j+|\beta|}{k}\frac{p-1}2}.
\end{align*}
Estimate \eqref{e:29} follows.
Moreover, by the definition of $w_j$ and setting $b=  \kappa^{\frac{p-1}2} \frac{p-1}{3p+1}$,
\begin{equation}\label{e:dtwj}
\partial_tw_j = - b
\left(\partial_t(U_0^{\frac {p+1}2}) \int_0^t U_0^{-p} \Ens_{j-1} ds  + \partial_t(U_0^{-p}) \int_{t}^{t_{j-1}} U_0^{\frac {p+1}2} \Ens_{j-1} ds \right) .
\end{equation}
Similarly as above, Leibniz's formula, \eqref{e:20}, and~\eqref{e:35}-\eqref{e:36} yield~\eqref{e:30}.
  Note that we have proved estimates~\eqref{e:29} and~\eqref{e:30} for all $0<t\le t _{ j-1 }$.

\smallskip
Proof of \eqref{e:31}-\eqref{e:31bis}.
For $0<t\leq t _{ j-1 } $ and $|x|\leq 2$,
by the estimate \eqref{e:20} on $w_j$ for $\beta=0$, the property $U_0\gtrsim 1$ for $ |x| \le 2$, and the definition of $\chi_j$, we have
\begin{equation*}
\chi_j |w_j|\lesssim \chi_j U_0^{1-j(1-\frac 1k)(p-1)} 
\lesssim \chi_j U_0^{1-(1-\frac 1k)(p-1)} 
\lesssim \chi_j (t+A)^{2-\frac 2k} U_0 
\lesssim (t +a_j)U_0 .
\end{equation*}
Choosing $0<a_j\leq 1$ and $0<t_j\leq t_{j-1}$ sufficiently small, for all $t\in (0,t_{j}]$,
\[
\chi_j |w_j| \leq 2^{-j-2} U_0 \quad \hbox{and} \quad 
\chi_j |w_j| \leq 2^{-j} (1+U_0)^{-\frac {p-1}2} U_0 .
\]
From now on, $a_j$ and $t_j$ are fixed to such values.
In the case $j=1$, this proves \eqref{e:31} for $|x|\leq 2$. For $2\leq j\leq J$, combining this estimate with \eqref{e:31} for $j-1$,
we find, for all $t\in (0,t_j]$ and $|x|\leq 2$,
\begin{equation}\label{point}
\sum_{\ell =1}^{j} \chi_\ell  |w_\ell |\leq \frac 14 (1-2^{-j}) U_0
\quad \hbox{and} \quad 
\sum_{\ell =1}^{j} \chi_\ell  |w_\ell |\leq (1-2^{-j})(1+U_0)^{-\frac {p-1}2} U_0,
\end{equation}
which implies \eqref{e:31} for $U_j$ and for $|x|\leq 2$. 
 
To prove \eqref{e:31bis} for $|x|\leq 2$, we note that by \eqref{e:30} with $\beta =0$ and $U_0\gtrsim 1$, 
\[
 \sum_{\ell =1}^{j} \chi_\ell |\partial_t w_\ell |\lesssim \sum_{\ell =1}^j\chi_\ell  U_0^{\frac{p+1}2-\ell (1-\frac 1k)(p-1)}
 =U_0  \sum_{\ell =1}^j\chi_\ell  U_0^{\frac{p-1}2 ( 1- 2\ell (1-\frac 1k) )}
\lesssim U_0.
\]

For $|x|\geq 2$, \eqref{on:A} implies that $A(x)\geq 2^k\geq 2a_1\geq\cdots\geq 2a_j$ and thus
$\chi_j(x)=0$ and $U_j(t,x)=U_0(t,x)$. The same applies to $\partial_t U_j$.

\smallskip
Proof of \eqref{e:32}.
Differentiating \eqref{e:dtwj} with respect to $t$, using the relations \eqref{e:19}, 
$\partial_{tt} (U_0^{\frac{p+1}2})=f'(U_0) U_0^{\frac{p+1}2}$ and
$\partial_{tt} (U_0^{-p})=f'(U_0) U_0^{-p}$
(these calculations are related to observations made in Section \ref{sec:2.1}), we check that $w_j$ satisfies
\[
\partial_{tt} w_j=f'(U_0) w_j+\Ens_{j-1}.
\]
Using also $U_j=U_{j-1}+\chi_j w_j$ and the definition of $\Ens_{j-1}$, we obtain
\begin{align*}
\Ens_j
&=\Ens_{j-1} 
-\chi_j \partial_{tt}w_j +\Delta (\chi_j w_j)+ f(U_j)-f(U_{j-1})\\
&= (1-\chi_j) \Ens_{j-1}+\Delta (\chi_j w_j) +f(U_j)-f(U_{j-1})-f'(U_0)\chi_j w_j.
\end{align*}
We estimate $\partial_x^\beta$ of each term on the right-hand side above for $|\beta|\leq k-3-2j$ and $|x|\leq 2$.
For the first term, recall that for $x$ such that $A(x)\leq a_j$, it holds
$1-\chi_j(x)=0$ and for any $\beta$, $\partial_x^{\beta} \chi_j(x)=0$.
Moreover, for $0<t\leq 1$, for $0\leq x\leq 2$ such that $A(x)\geq a_j$, it holds
$A(x)\approx 1$ and so $U_0(t,x)\approx 1$. Thus, using the Leibniz formula and \eqref{e:32} for $\Ens_{j-1}$, we find
\[
|\partial_x^\beta[(1-\chi_j) \Ens_{j-1}]|
\lesssim U_0^{1-(j-1)(p-1)+\frac{2j +|\beta|}k \frac{p-1}2}\lesssim U_0^{1-j(p-1)+\frac{2j+2 +|\beta|}k \frac{p-1}2}.
\]
Next, by Leibniz's formula, the properties of $\chi$ and $\chi_j$, the estimate \eqref{e:29} on $w_j$
and then $U_0\gtrsim 1$, we have, for $0<t<t_j$ and $|x|\leq 2$,
\[
|\partial_x^\beta\Delta (\chi_j w_j)|\lesssim 
\sum_{|\beta'|\leq |\beta|+2} |\partial_x^{\beta'} w_j|
\lesssim U_0^{1-j(p-1)+\frac{2j + 2+|\beta|}k \frac{p-1}2}.
\]
Last, we estimate $\partial_x^\beta[f(U_j)-f(U_{j-1})-f'(U_0)\chi_j w_j]$.
We begin with the case $\beta=0$.
Recall that by \eqref{e:31}, we have $0<\frac 34 U_0\leq U_j\leq \frac 54 U_0$, so that by elementary calculations
\[
\left|f(U_j)-f(U_{j-1})-f'(U_{j-1})\chi_j w_j\right|\lesssim  \chi_j U_0^{p-2}w_j^2
\]
and
\[
\left|f'(U_{j-1})-f'(U_0)\right|\lesssim U_0^{p-2} \sum_{\ell =1}^{j-1} \chi_\ell |w_\ell |.
\]
These estimates imply
\[
\left|f(U_j)-f(U_{j-1})-f'(U_0)\chi_j w_j\right|\lesssim
\chi_j U_0^{p-2} |w_j|\sum_{\ell =1}^{j}\chi_\ell  |w_\ell |.
\]
For $1\leq \ell \leq j$, using \eqref{e:29} and $U_0\gtrsim 1$, we have
\begin{align*}
U_0^{p-2}|w_j||w_\ell |&\lesssim U_0^{p-2} U_0^{1-j(1-\frac 1k)(p-1)}
U_0^{1-\ell (1-\frac 1k)(p-1)}\lesssim U_0^{p-(j+\ell )(1-\frac 1k)(p-1)}\\
&
\lesssim U_0^{p-(j+1)(1-\frac 1k)(p-1)}\lesssim U_0^{1-j(1-\frac 1k)(p-1)+\frac 1k(p-1)}.
\end{align*}
Thus, 
$|f(U_j)-f(U_{j-1})-f'(U_0)\chi_j w_j |\lesssim U_0^{1-j(1-\frac 1k)(p-1)+\frac 1k(p-1)}$ is proved.

Now, we deal with the case $1\leq |\beta|\leq k-3-2j$. By the Taylor formula with integral remainder, we have, for any $U$
and $w$,
\[
f(U+w)-f(U)-f'(U)w=w^2\int_0^1 (1-\theta)f''(U+\theta w) d\theta.
\]
Thus, by Leibniz's formula \eqref{lbz0}
\begin{multline*}
\partial_x^\beta\left[ f(U+w)-f(U)-f'(U)w\right]
\\=\sum_{\beta'\leq \beta} \binom{\beta}{\beta'} \left( \partial_x^{\beta-\beta'}(w^2)\right)
\int_0^1(1-\theta)\partial_x^{\beta'}[f''(U+\theta w)] d\theta.
\end{multline*}
Moreover, by the Faa di Bruno formula \eqref{fdb0}, for $\beta'\neq 0$, denoting
$n'=|\beta'|$,
\[
\partial_x^{\beta'} [f''(U+\theta w)]
=\sum_{\nu=1}^{n'}f^{(\nu+2)}(U+\theta w)\sum_{P(\beta',\nu)}(\beta'!)
\prod_{\ell =1}^{n'} \frac{(\partial_x^{\beta_\ell }(U+\theta w))^{\nu_\ell }}{(\nu_\ell !)(\beta_\ell !)^{\nu_\ell }}.
\]
To estimate the term $\partial_x^\beta[f(U_j)-f(U_{j-1})-f'(U_0)\chi_j w_j]$, we apply these formulas to $U=U_j$ and $w=\chi_j w_j$. For $\beta'\leq \beta$, using \eqref{e:29} and the properties of $\chi$, we have
\begin{align*}
\left|\partial_x^{\beta-\beta'}\left[(\chi_jw_j)^2\right]\right|
&\lesssim \sum_{\beta''\leq \beta-\beta'}
\left|\partial_x^{\beta''} (\chi_jw_j)\right|
\left|\partial_x^{\beta-\beta'-\beta''}(\chi_jw_j)\right|\\
&\lesssim U_0^{2-2j(p-1)+\frac{4j+|\beta|-|\beta'|}{k}\frac{p-1}2}.
\end{align*}
For $\beta'=0$ and $\theta\in [0,1]$, using also \eqref{point}, we obtain
\begin{align*}
\left|\left( \partial_x^{\beta}[(\chi_jw_j)^2]\right) f''(U_0+\theta \chi_jw_j)\right|
&\lesssim U_0^{2-2j(p-1)+\frac{4j+|\beta|}{k}\frac{p-1}2}U_0^{p-2}\\
&\lesssim U_0^{1-j(p-1)+\frac{2j+2+|\beta|}{k}\frac{p-1}2}.
\end{align*}
For $\beta'\neq 0$, $\beta'\leq \beta$ and $\theta\in [0,1]$, using \eqref{e:20}, \eqref{e:29} and \eqref{point},
we have (recall that the definition of $P(\beta',\nu)$ implies that $\sum_{\ell =1}^{n'}\nu_\ell =\nu$ and
$\sum_{\ell =1}^{n'}\nu_\ell |\beta_\ell |=|\beta'|$)
\begin{align*}
|\partial_x^{\beta'} [f''(U_{j-1}+\theta \chi_j w_j)]|
&\lesssim \sum_{\nu=1}^{n'} U_0^{p-\nu-2} \sum_{P(\beta',\nu)} 
\prod_{\ell =1}^{n'} \left(U_0^{1+\frac{|\beta_\ell |}{k}\frac{p-1}2}\right)^{\nu_\ell }\\
&\lesssim \sum_{\nu=1}^{n'} U_0^{p-\nu-2}U_0^{\nu+ \frac{|\beta'|}{k} \frac{p-1}2}
\lesssim U_0^{p-2+\frac{|\beta'|}{k}\frac{p-1}2}.
\end{align*}
Thus, similarly as before, it holds
\begin{equation*}
\left|\partial_x^{\beta-\beta'}\left[(\chi_j w_j)^2\right] \partial_x^{\beta'}\left[f''(U_{j-1}+\theta \chi_j w_j)\right]\right|
\lesssim U_0^{1-j(p-1)+\frac{2j+2+|\beta|}{k}\frac{p-1}2}.
\end{equation*}
Integrating these estimates in $\theta\in [0,1]$, we obtain
\begin{equation}\label{UN}
\left|\partial_x^{\beta} [f(U_j)-f(U_{j-1})-f'(U_{j-1})\chi_j w_j]\right| 
 \lesssim U_0^{1-j(p-1)+\frac{2j+2+|\beta|}{k}\frac{p-1}2}.
\end{equation}

\smallskip

By similar arguments, for any $U,W,w$, we have
\begin{equation*}
f'(U)-f'(W)=(U-W) \int_0^1 f''(W+\theta(U-W)) d\theta,
\end{equation*}
and thus
\begin{multline*}
\partial_x^{\beta}[w(f'(U)-f'(W))]
\\=\sum_{\beta'\leq \beta} \binom{\beta}{\beta'} \left(\partial_x^{\beta-\beta'}[w(U-W)]\right)
\int_0^1 \partial_x^{\beta'} [f''(W+\theta(U-W)) ] d\theta.
\end{multline*}
Moreover, for $\beta'\neq 0$,
\begin{multline*}
\partial_x^{\beta'} [f''(W+\theta(U-W)) ]\\ =
 \sum_{\nu=1}^{n'} f^{(\nu+2)}(W+\theta(U-W))\sum_{P(\beta',\nu)}(\beta'!)
\prod_{\ell =1}^{n'} \frac{\left(\partial_x^{\beta_\ell }(W+\theta(U-W))\right)^{\nu_\ell }}{(\nu_\ell !)(\beta_\ell !)^{\nu_\ell }}.
\end{multline*}
To estimate the term $\partial^\beta [\chi_j w_j(f'(U_{j-1})-f'(U_0))]$,
we apply these formulas to $U=U_{j-1}$, $W=U_0$ and $w=\chi_j w_j$.

For $\beta'\leq \beta$, using \eqref{e:29} and the properties of $\chi$, we have, for $1\leq \ell \leq j-1$,
\begin{equation*}
\left| \partial_x^{\beta-\beta'}[\chi_j w_j \chi_\ell w_\ell ]\right|
\lesssim U_0^{2-(j+\ell )(p-1)+\frac{2j+2\ell +|\beta|-|\beta'|}{k}\frac{p-1}2}.
\end{equation*}
For $\beta'=0$ and $\theta\in [0,1]$, from \eqref{point}, we obtain
\begin{equation*}
\left|\partial_x^{\beta}\left[\chi_j w_j(U_{j-1}-U_0)\right] f''(U_0+\theta (U_{j-1}-U_0))\right|
\lesssim U_0^{1-j(p-1)+\frac{2j+2+|\beta|}{k}\frac{p-1}2}.
\end{equation*}
For $\beta'\neq 0$, $\beta'\leq \beta$ and $\theta\in [0,1]$, by the formula above, using \eqref{e:20}, \eqref{e:29} and \eqref{point}, we have as before
\begin{align*}
|\partial_x^{\beta'}[f''(U_0+\theta (U_{j-1}-U_0))]|
&\lesssim \sum_{\nu=1}^{n'} U_0^{p-\nu-2} \sum_{P(\beta',\nu)} 
\prod_{\ell =1}^{n'} \left(U_0^{1+ \frac{|\beta_\ell |}{k}\frac{p-1}2}\right)^{\nu_\ell }\\
&\lesssim \sum_{\nu=1}^{n'} U_0^{p-\nu-2} U_0^{\nu+\frac{|\beta'|}{k}\frac{p-1}2}
\lesssim U_0^{p-2+ \frac{|\beta'|}{k}\frac{p-1}2}.
\end{align*}
Thus, we obtain
\begin{equation*}
\left|\partial_x^{\beta-\beta'}\left[\chi_j w_j(U_{j-1}-U_0)\right] \partial_x^{\beta'}\left[f''(U_0+\theta (U_{j-1}-U_0))\right]\right|
\lesssim U_0^{1-j(p-1)+\frac{2j+2+|\beta|}{k}\frac{p-1}2}.
\end{equation*}
Integrating in $\theta\in [0,1]$ and summing in $\beta'\leq \beta$, we obtain
\begin{equation}\label{DEUX}
\left| \partial^\beta [\chi_j w_j(f'(U_{j-1})-f'(U_0))]\right|
\lesssim U_0^{1-j(p-1)+\frac{2j+2+|\beta|}{k}\frac{p-1}2}.
\end{equation}
Combining \eqref{UN} and \eqref{DEUX}, we have proved for $t\in (0,t_{j}]$, $|x|\leq 2$,
\begin{equation*}
\left|\partial^\beta [f(U_j)-f(U_{j-1})-f'(U_0)\chi_j w_j]\right|
\lesssim U_0^{1-j(p-1)+\frac{2j+2+|\beta|}{k}\frac{p-1}2}.
\end{equation*}

In conclusion, we have estimated all terms in the expression of $\partial_x^\beta\Ens_j$ and \eqref{e:32} is now proved.

\smallskip

Finally, for $|x|\geq 2$, \eqref{on:A} implies that $A(x)\geq 2^k\geq 2a_1\geq\cdots\geq 2a_j$ and thus
$\chi_j(x)=0$, $U_j(t,x)=U_0(t,x)$ and $\Ens_{j}(t,x)= \Ens_0(t,x)$,
so that~\eqref{e:33} follows from~\eqref{e:22}.
\end{proof}
\section{Uniform bounds on approximate solutions}\label{sec:3}
 Let the  function $\chi $ be given by~\eqref{def:chi} and $U_J$ be defined as in \S\ref{s23} with $J$ and $k$ as in \eqref{def:Jk}.
Set
\begin{equation}\label{def:la}
\lambda= \min \left(J-\frac{2}{p-1},\frac 12\right) \in \left(0,\frac 12\right],
\end{equation}
and impose the following additional condition on $k$
\begin{equation}\label{on:k}
k\geq \frac{2(p+1)}{\lambda(p-1)} +2.
\end{equation}
For any $n$ large, let $T_n=\frac 1n<t_J$ and
\begin{equation}\label{def:Bn}
B_n=\sup_{t\in [T_n,t_J]} \|U_J(t)\|_{L^\infty}\quad \mbox{so that}\quad
\lim_{n\to \infty} B_n=\infty.
\end{equation}
We let $n$ be sufficiently large so that $B_n>1$, and 
we define the function $f_n:\R\to [0,\infty)$  by
\begin{equation}\label{def:fn}
f_n(u )= f(u) \chi\left(\frac{u}{B_n}\right)
\quad \mbox{so that}\quad f_n(u)=\begin{cases} f(u) & \mbox{for $|u|<B_n$}\\ 0 & \mbox{for $u>2B_n$} \end{cases}.
\end{equation}
Let $F_n(v)=\int_0^v f_n(w) dw$.
It follows from elementary calculations that for every $\alpha\in \N$, there exists a constant $C_\alpha>0$ independent of $n$, such that for all $u>0$,
\begin{equation}
|f_n^{(\alpha)}(u)|\leq C_\alpha u^{p-\alpha}.
\end{equation}
In particular, we observe that Taylor's estimates such as \eqref{taylor0}--\eqref{taylor} still hold for $F_n$ and $f_n$
with constants independent of $n$. 
We will refer to these inequalities for $F_n$ and $f_n$ with the same numbers~\eqref{taylor0}, \eqref{taylor1} and~\eqref{taylor}. 
In this proof, any implicit constant related the symbol $\lesssim$ is independent of $n$.

We define the sequence of solutions $u_n$ of
\begin{equation}\label{eq:un}
\left\{\begin{aligned}&\partial_{tt} u_n-\Delta u_n=f_n(u_n)\\
&u_n(T_n)=U_J(T_n),\quad \partial_t u_n(T_n)=\partial_t U_J(T_n).
\end{aligned}\right.\end{equation}
The nonlinearity $f_n$ being globally Lipschitz, the existence of a global solution $(u_n,\partial_t u_n)$ in the energy space is a consequence of standard arguments from semi-group theory.
Using energy estimates, we prove  uniform bounds on $u_n$ in the energy space.
For this we set, for all $t\in [T_n,t_J]$,
\begin{equation}\label{uneps}
u_n(t)=U_J(t)+\varepsilon_n(t),
\end{equation}
so that $(\varepsilon_n,\partial_t\varepsilon_n)\in \cont([T_n,t_J],H^1(\R^N)\times L^2(\R^N))$.

\begin{prop}\label{pr:2}
There exist $C>0$, $n_0>0$ and $0<\delta_0<1$ such that 
\begin{equation}\label{unif}
\|(\varepsilon_n(t),\partial_t \varepsilon_n(t))\|_{H^1\times L^2}\leq C (t-T_n)^{\frac\lambda2}.\end{equation}
for all $n\geq n_0$ and $t\in [T_n, T_n+\delta_0]$, where $\lambda $ is given by~\eqref{def:la}. 
\end{prop}

\begin{proof}
The equation of $\varepsilon_n$ on $[T_n,t_J]\times \R^N$ is
\begin{equation}\label{eq:wn}
\left\{\begin{aligned}&\partial_{tt} \varepsilon_n-\Delta \varepsilon_n=f_n(U_J+\varepsilon_n)-f_n(U_J) + \Ens_J\\
&\varepsilon_n(T_n)=0,\quad \partial_t \varepsilon_n(T_n)=0
\end{aligned}\right.\end{equation}
where we have used  from \eqref{def:Bn} and \eqref{def:fn} that  $f(U_J)=f_n(U_J)$ on $[T_n,t_J]\times \R^N$.

Define the auxiliary function $z$ as follows
\[
\varepsilon_n= Q^{\frac 12} z \quad \hbox{where} \quad Q=(1-\chi+U_0)^{p+1} ,
\]
 where, by abuse of notation, we denote $\chi (x) = \chi (  |x| )$.
We note that $Q \gtrsim 1$, $Q \lesssim t^{- \frac {2 (p+1) } {p-1}}$. Moreover, it follows from~\eqref{e:20}  that $ |\nabla U_0 | \lesssim Q^{ \frac {1} {p+1} + \frac {p-1} {2k (p+1)} }$, from which we deduce easily that
$ |\nabla Q |\lesssim t^{-\frac {1} {k}} Q$. One proves similarly that $ |\Delta  Q |\lesssim t^{-\frac {2} {k}} Q$. 
To write the equation of $z$, we compute
\begin{align*}
\partial_{tt} \varepsilon_n & =\partial_{tt}(Q^{\frac 12}z)=(\partial_{tt} Q^{\frac 12}) z + (p+1) (1-\chi+U_0)^{\frac {p-1}2} \partial_t U_0 \partial_t z
+Q^{\frac 12} \partial_{tt} z\\
&= (\partial_{tt} Q^{\frac 12}) z + Q^{-\frac 12} \partial_t (Q \partial_t z) .
\end{align*}
Thus, setting $G=Q^{\frac 12} (f'(U_0) Q^{\frac 12} - \partial_{tt} Q^{\frac 12})$, we obtain
\begin{equation} \label{eqforz} 
\begin{split} 
\partial_t (Q\partial_t z) =& Q^{\frac 12} \Delta (Q^{\frac 12} z) 
\\
& + Q^{\frac12} \left( f_n(U_J+Q^{\frac 12} z)-f_n(U_J)-f_n'(U_0)Q^{\frac 12} z\right) 
+ Gz+Q^{\frac 12}\Ens_J.
\end{split} 
\end{equation} 
Let $\sigma=\frac 34$.
We define the following weighted norm and energy functional for $z$,
\begin{align*}
\mathcal N&= \left(\int (Q\partial_t z)^2+ Q^2 |\nabla z|^2 + t^{-2\sigma} Q^2 z^2\right)^{\frac 12},\\
\mathcal H&= \int \Big[ (Q\partial_t z)^2+ Q^2 |\nabla z|^2 + t^{-2\sigma} Q^2 z^2
\\&\quad - Q \left(2F_n(U_J+Q^{\frac 12} z) -2F_n(U_J)-2F_n'(U_J)Q^{\frac 12} z- F_n''(U_0)Qz^2\right)\Big].
\end{align*}
We remark that the first two terms in $\mathcal H$ are the energy for the linear part of equation~\eqref{eqforz}. 
The third term yields the control of a weighted $L^2$ norm, and the last term is associated with the nonlinear terms in the equation. 

\textbf{Step 1.} Coercivity of the energy. We claim that, for $0<\delta\leq t_J$  and $0<\omega\leq 1$ sufficiently small,
for $n$ large,  if $\mathcal N\leq \omega$  and $T_n\leq t\leq \delta$ then
\begin{equation}\label{coer}
\mathcal N^2 \leq 2 \mathcal H,
\end{equation}
and
\begin{equation}\label{coer2}
\|(\varepsilon_n(t),\partial_t \varepsilon_n(t))\|_{\dot H^1\times L^2}\lesssim \mathcal N,\quad
\|\varepsilon_n(t)\|_{L^2}\lesssim t^\sigma  \mathcal N.
\end{equation}

\emph{Proof of \eqref{coer}.}
Let
\begin{equation} \label{fDfnLam} 
\GDT1 = | 2 F_n(U_J+Q^{\frac 12} z) - 2 F_n(U_J)- 2 F_n'(U_J)Q^{\frac 12} z-  F_n''(U_0)Qz^2 | .
\end{equation} 
The triangle inequality and the Taylor inequality \eqref{taylor0} yield
\begin{equation} \label{fEstLam0} 
\begin{split} 
\GDT1 \lesssim & | 2 F_n(U_J+Q^{\frac 12} z) - 2 F_n(U_J)- 2 F_n'(U_J)Q^{\frac 12} z- F_n''(U_J)Qz^2 |
\\ & +|F_n''(U_0)-F_n''(U_J)| Q z^2 \\  \lesssim & \GCT1
\end{split} 
\end{equation} 
where
\begin{equation} \label{fDefLam2} 
 \GCT1 = Q^{\frac{p+1}2}|z|^{p+1}+U_J^{p-\bar p}Q^{\frac{\bar p+1}{2}}|z|^{\bar p+1}
+U_0^{p-2}|U_0-U_J|Qz^2 .
\end{equation} 
Using $U_J\lesssim U_0$ and $ U_0 \lesssim Q^{\frac {1} {p+1}}$, we see that $U_J^{p-\bar p} \lesssim Q^{ \frac {p- \bar p} {p+1} }$. Moreover, since $|U_0-U_J|\lesssim (1+U_0)^{-\frac{p-1}2}U_0\lesssim Q^{-\frac{p-1}{2(p+1)}} U_0$ (see \eqref{e:31}), we obtain 
\begin{equation*} 
U_0^{p-2}|U_0-U_J|  \lesssim U_0^{p-1} Q^{-\frac{p-1}{2(p+1)}} \lesssim Q^{ \frac{p-1}{2(p+1)}} ,
\end{equation*} 
and so
\begin{equation} \label{fDefLam3} 
 \GCT1 \lesssim Q^{\frac{p+1}2}|z|^{p+1}+ Q^{\frac{\bar p+1}{2} + \frac {p- \bar p} {p+1} }|z|^{\bar p+1}
+  Q^{ \frac{p-1}{2(p+1)} + 1} z^2 .
\end{equation} 
It follows that
\begin{equation*} 
\int Q \GCT1 \lesssim \int Q^{\frac{p+3}2}|z|^{p+1}+\int Q^{\frac{\bar p+3}2+\frac{p-\bar p}{p+1}}|z|^{\bar p+1}
 +\int Q^{\frac{p-1}{2(p+1)}+2} z^2. 
\end{equation*} 
For the first term on the right-hand side above, we use $Q \gtrsim 1$, thus
\begin{equation} \label{fEstIp1} 
\int Q^{\frac {p+3}2} |z|^{p+1} =\int Q^{-\frac {p-1}2} |Q z|^{p+1} \lesssim \int |Qz|^{p+1} . 
\end{equation} 
Applying now~\eqref{gagl},  $|\nabla Q|\lesssim t^{-\frac 1k}Q$, and the definition of $\mathcal N$,
\begin{equation} \label{festQz} 
\begin{split} 
\int |Qz|^{p+1} &  \lesssim 
\left(\int |\nabla (Qz)|^2\right)^{\frac{N}4(p-1)}\left(\int Q^2z^2\right)^{ \frac {p+1} {2} - \frac {N} {4} (p-1) } 
\\ & 
\lesssim \left(\int Q^2|\nabla z|^2 + t^{-\frac 2k} \int Q^2 z^2\right)^{\frac{N}4(p-1)}\left(\int Q^2 z^2\right)^{ \frac {p+1} {2} - \frac {N} {4} (p-1)  } \\ & \lesssim \mathcal N^{p+1} .
\end{split} 
\end{equation} 
In the case $1<p\leq 2$, one has $\bar p=p$ and the second term is identical to the first one.
In the case $p>\bar p=2$, the second term $Q^{\frac52+\frac{p-2}{p+1}}|z|^{3}=Q^{\frac{7p+1}{2(p+1)}}|z|^3$ is estimated as follows (using $ |z|^3 \lesssim a^{p-2}  |z|^{p+1} + \frac {1} {a} z^2$ with $a= Q^{\frac {p-1} {p+1}}$, 
 $Q^{-1}\lesssim 1$ and $Q^{\frac{p-1}{2(p+1)}}\lesssim t^{-1}$)
\begin{align*}
Q^{\frac{7p+1}{2(p+1)}}|z|^3&\lesssim
 Q^{-\frac 32 \frac{p-1}{p+1}} Q^{p+1} |z|^{p+1} + Q^{\frac{5p+3}{2(p+1)}} z^2\\
 &\lesssim Q^{p+1}|z|^{p+1}+ Q^{\frac{p-1}{2(p+1)}}Q^2z^2\lesssim Q^{p+1}|z|^{p+1}+t^{-1}Q^2z^2,
\end{align*}
and so
\begin{equation} \label{fEstIp2} 
\int Q^{\frac{\bar p+3}2+\frac{p-\bar p}{p+1}}|z|^{\bar p+1} \lesssim \mathcal N^{p+1}+t^{2\sigma -1} \mathcal N^2.
\end{equation}
Last, since $Q^{\frac{p-1}{2(p+1)}}\lesssim t^{-1}$, we observe that
\[
\int Q^{\frac{p-1}{2(p+1)}} Q^2z^2\lesssim t^{-1} \int Q^2 z^2\lesssim t^{2\sigma -1} \mathcal N^2.
\]
In conclusion, we have obtained $\int Q \GDT1 \lesssim \int Q \GCT1 \lesssim t^{2\sigma-1} \mathcal N^2 + \mathcal N^{p+1}$, which implies that for $t$  and $\mathcal N$ small enough, $\mathcal H\geq \frac 12 \mathcal N^2$.
 
\emph{Proof of \eqref{coer2}.}
Since $\varepsilon_n=Q^{\frac 12}z$, the inequality $\|\varepsilon_n\|_{L^2}\lesssim t^{-1}\mathcal N$ follows readily from the definition of $\mathcal N$ and $Q\gtrsim 1$.
Next, using $|\nabla Q|\lesssim t^{-\frac 1k} Q$, we see that 
\[
\int |\nabla \varepsilon_n|^2=\int \Big| Q^{\frac 12}\nabla z+\frac 12 Q^{-\frac 12}z\nabla Q\Big|^2
\lesssim \int Q |\nabla z|^2 + t^{-\frac 2k} \int Q z^2 \lesssim \mathcal N^2.
\]
Last, using $|\partial_t Q|\lesssim Q^{\frac{p-1}{2(p+1)}} Q\lesssim t^{-\frac12} Q^{\frac{5p+3}{4(p+1)}}$, we have
\begin{multline*}
\int |\partial_t\varepsilon_n|^2=\int \Big| Q^{\frac 12}\partial_t z+\frac 12 Q^{-\frac 12}z \partial_t Q\Big|^2
\lesssim \int Q |\partial_t z|^2+\int |\partial_t Q|^2 Q^{-1} z^2\\
\lesssim \int Q^2 |\partial_t z|^2+ t^{-1} \int Q^{\frac{3p+1}{2(p+1)}} z^2\lesssim \mathcal N^2.
\end{multline*}
This completes the proof of \eqref{coer2}.

\medskip
 
\textbf{Step 2.} Energy control. 
We claim that for $0<\delta\leq t_J$ small enough and $C>0$ large enough,  for any $n$ large and for all $t\in [T_n,T_n+\delta]$
\begin{equation}\label{dtener}
\frac d{dt} \mathcal H \leq C \left[t^{-1+\lambda} \mathcal N+ t^{-\frac 12} \mathcal N^2 +\mathcal N^{p+1}\right].
\end{equation}

\emph{Proof of \eqref{dtener}.}
Taking the time-derivative of all the terms in $\mathcal H$, we obtain
\begin{align*}
\frac 12 \frac d{dt} \mathcal H&=\int \left( Q\partial_t z \partial_t(Q\partial_t z) + Q^2 \nabla z\cdot \nabla \partial_t z 
+ t^{-2\sigma} Q^{2} z \partial_t z\right)\\&\quad
- \int Q^{\frac 32} \left(f_n(U_J+Q^{\frac 12} z) -f_n(U_J)-f_n'(U_0)Q^{\frac 12}z\right)\partial_t z
\\&\quad +\int Q\partial_t Q |\nabla z|^2 + t^{-2\sigma}\int Q\partial_t Qz^2
-\sigma t^{-2\sigma-1}\int Q^2 z^2
\\&\quad-\frac 12 \int \partial_t Q \left(2F_n(U_J+Q^{\frac 12} z)-2F_n(U_J)-2F_n'(U_J)Q^{\frac 12}z-F_n''(U_0)Qz^2\right)
\\&\quad-\frac 12 \int \partial_t Q\left(f_n(U_J+Q^{\frac 12}z)-f_n(U_J)-f_n'(U_0)Q^{\frac 12}z\right)Q^{\frac 12}z
\\&\quad- \frac 12 \int Q \partial_t U_0\left( 2f_n(U_J+Q^{\frac 12}z)-2f_n(U_J)-2f_n'(U_J)Q^{\frac 12} z-f_n''(U_0)Q z^2\right)\\
&\quad- \frac 12 \int Q \partial_t (U_J-U_0)\left( 2f_n(U_J+Q^{\frac 12}z)-2f_n(U_J)-2f_n'(U_J)Q^{\frac 12} z\right)\\
&=I_1+I_2+I_3+I_4+I_5+I_6+I_7.
\end{align*}
First, we note that $\partial_tQ=(p+1)Q^{\frac p{p+1}} \partial_t U_0  \le 0$, so that
\begin{equation*} 
I_3 \le  -\sqrt{2(p+1)} \int U_0^{\frac {p+1}2} Q^{\frac{2p+1}{p+1}}|\nabla z|^2
- \sigma t^{-2\sigma-1} \int Q^{2} z^2.
\end{equation*} 
We now use equation~\eqref{eqforz} to replace the term $ \partial_t(Q\partial_t z) $ in $I_1$, and we obtain
\begin{align*}
I_1+I_2 &=
\int \left(Q^{\frac 32} \partial_t z \Delta (Q^{\frac 12} z)+Q^2 \nabla z\cdot \nabla \partial_t z\right)
\\&\quad + \int \left(Gz+Q^{\frac 12}\Ens_J \right) Q\partial_t z
+ t^{-2\sigma} \int Q^2 z \partial_t z
=I_8+I_9+I_{10}.
\end{align*}
The term $I_{10}$ is controlled using the Cauchy-Schwarz inequality,
\[|I_{10}|\leq \frac1{10} {|I_3|} + C t^{-2\sigma+1} \int (Q\partial_t z)^2
\leq \frac1{10}{|I_3|}+C t^{-2\sigma+1} \mathcal N^2
\leq \frac1{10}{|I_3|}+C t^{-\frac 12} \mathcal N^2
.\]
Next, integrating by parts,
\begin{align*}
I_8 & =-\int \nabla (Q^{\frac 12} z) \cdot \nabla (Q^{\frac 32} \partial_t z)+Q^2 \nabla z\cdot \nabla (\partial_t z)\\
& =-\int z\nabla (Q^{\frac 12}) \cdot \nabla (Q^{\frac 32} \partial_t z ) -\int Q^{\frac 12}(\partial_t z) \nabla z \cdot \nabla (Q^{\frac 32})\\
& =-\int Q \partial_t z\nabla z \cdot \nabla Q +\int \Delta (Q^{\frac 12}) Q^{\frac 32} z \partial_t z .
\end{align*}
By $|\nabla Q|\lesssim t^{-\frac 1k} Q$ and the Cauchy-Schwarz inequality,
\[
\left| \int Q \partial_t z\nabla z \cdot \nabla Q\right|
\lesssim t^{-\frac 1k} \mathcal N^2.
\]
Similarly, $|\Delta (Q^{\frac 12}) Q^{\frac 32}|\lesssim |\nabla Q|^2 +|\Delta Q|Q\lesssim t^{-\frac 2k} Q^2$, and so
\[
\left|\int \Delta (Q^{\frac 12}) Q^{\frac 32} z \partial_t z \right|\leq 
\int Q^2|\partial_t z|^2 + t^{-\frac 4k}\int Q^2z^2 \lesssim \mathcal N^2.
\]

We note that by Cauchy-Schwarz,
\[
|I_9|\lesssim \|G\|_{L^\infty}\mathcal N^2 +\|Q^{\frac 12} \Ens_J\|_{L^2} \mathcal N,
\]
and so, we only have to bound the $L^\infty$ norm of $G$ and the $L^2$ norm of $Q^{\frac 12} \Ens_J$.
We begin with
$G= Q^{\frac 12} \left( pU_0^{p-1}Q^{\frac 12}-\partial_{tt} Q^{\frac 12}\right)$.
Using $Q=(1-\chi+U_0)^{p+1}$ and the expressions of $\partial_{tt}U_0$ and $(\partial_t U_0)^2$, we observe that
\[
\partial_{tt} Q^{\frac 12}= \frac {p+1}2U_0^pQ^{\frac {p-1}{2(p+1)}}
+\frac {p-1}2U_0^{p+1}Q^{\frac{p-3}{2(p+1)}}.
\]
Thus,
\begin{align*}
&pU_0^{p-1}Q^{\frac 12}-\partial_{tt} Q^{\frac 12}
\\&\quad =\frac {p+1}2U_0^{p-1}Q^{\frac {p-1}{2(p+1)}}\left(Q^{\frac 1{p+1}}-U_0\right)
+\frac {p-1}2U_0^{p-1}Q^{\frac{p-3}{2(p+1)}}\left(Q^{\frac 2{p+1}}-U_0^2\right)\\
&\quad=\frac {p+1}2U_0^{p-1}Q^{\frac {p-1}{2(p+1)}}(1-\chi)+\frac {p-1}2U_0^{p-1}Q^{\frac{p-3}{2(p+1)}}(1-\chi)\left(1-\chi+2U_0\right).
\end{align*}
Since for $|x|>1$, we have $U_0\lesssim 1$ and $Q\lesssim 1$, we obtain 
  $\|G\|_{L^\infty}\lesssim 1$.

Now, we estimate $\|Q^{\frac 12}\Ens_J\|_{L^2}$ from Lemma~\ref{le:1}.
For $|x|\geq 2$, it follows from \eqref{e:33} that
\[
Q^{\frac 12}|\Ens_J|\lesssim |\Ens_J|\lesssim |x|^{-\frac{2k}{p-1}-2}.
\]
Note that for $N\geq 3$, $p-1\leq \frac 4{N-2}$ and so $\frac{2k}{p-1}+2\geq \frac{N-2}2k+2\geq N$.
Thus, the following bound holds $\|Q^{\frac 12}\Ens_J\|_{L^2(|x|>2)}\lesssim 1.$

Next, using \eqref{e:32}, we have for $|x|\lesssim 2$
\begin{equation*} 
\begin{split} 
Q^{\frac 12}|\Ens_J| & \lesssim Q^{\frac 12} U_0^{1-J(p-1)+\frac{(J+1) (p-1) }{k}}
\lesssim U_0^{\frac{p+3}2-J(p-1)+\frac{(J+1) (p-1)}{k} } \\ & 
\lesssim U_0^{\frac{p+3}2-J(1-\frac1k)(p-1)+\frac{p-1}k} .
\end{split} 
\end{equation*} 
Note that by~\eqref{def:la}, $J\ge \lambda + \frac {2} {p-1}$, so that $- J (1- \frac {1} {k}) + \frac {p-1} {k} \le - 2 (1- \frac {1} {k})+ \frac {p-1} {k}- \lambda (p-1) (1- \frac {1} {k}) $; and so 
\begin{equation*} 
Q^{\frac 12}|\Ens_J| \lesssim U_0^{\frac{p+3}2-2(1-\frac1k)+\frac{p-1}k - \lambda(p-1)(1-\frac1k)}
\lesssim U_0^{\frac{p-1}2+ \frac{p+1}k-\lambda(p-1)(1-\frac1k)} .
\end{equation*} 
Moreover, the   additional condition \eqref{on:k} is equivalent to
$\frac{p+1}{k}-\lambda(p-1)(1-\frac1k)\leq -\frac{\lambda(p-1)}{2}$.
Thus, for $|x|\lesssim 2$, 
\begin{equation} \label{fForarap} 
Q^{\frac 12}|\Ens_J|
\lesssim U_0^{ (1 - \lambda ) \frac{p-1}2}
\lesssim (t+A(x))^{-1+\lambda}\lesssim t^{-1+\lambda}.
\end{equation} 
Therefore, one obtains $\|Q^{\frac 12}\Ens_J\|_{L^2}\lesssim t^{-1+\lambda}$.
 
To complete the proof of \eqref{dtener}, we estimate $I_4$, $I_5$, $I_6$ and $I_7$.
First, using~\eqref{fDfnLam}--\eqref{fDefLam3}, and $  |\partial_t Q|\lesssim |\partial_t U_0| Q^{\frac{p}{p+1}}\lesssim U_0^{\frac{p+1}2} Q^{\frac{p}{p+1}} \lesssim Q^{\frac{ 3 p +1 }{2 (p+1) }} $, we obtain
\begin{equation*} 
\begin{split} 
 |\partial _t Q | \GDT1 &
\lesssim Q^{\frac{ 3 p +1 }{2 (p+1) }}  \GCT1 
 \\ & \lesssim Q^{ p+1 - \frac{ p(p-1) }{2 (p+1) }} |z|^{p+1}+ Q^{\frac{ 3 p +1 }{2 (p+1) } + \frac{\bar p+1}{2} + \frac {p- \bar p} {p+1} }|z|^{\bar p+1}
+    Q^{ 2 + \frac{p-1}{p+1} } z^2  .
\end{split} 
\end{equation*} 
Using $U_0 \gtrsim 1$ and the estimate~\eqref{festQz}, we treat the first term above as follows
\begin{equation*} 
\int Q^{ p+1 - \frac{ p(p-1) }{2 (p+1) }} |z|^{p+1} \lesssim \int | Q z|^{p+1}  \lesssim \mathcal N^{p+1} .
\end{equation*} 
In the case $1<p\leq 2$, one has $\bar p=p$ and the second term is identical to the first one.
In the case $p>\bar p=2$, the second term $Q^{\frac{ 4p }{ p+1 }}|z|^3$ is estimated as follows (using $ |z|^3 \lesssim a^{p-2}  |z|^{p+1} + \frac {1} {a} z^2$ with $a= Q^{\frac {p-1} {p+1}}$, 
 $Q^{-1}\lesssim 1$ and $Q^{\frac{p-1}{p+1}}\lesssim t^{-2}$)
\begin{equation*} 
Q^{\frac{ 4p }{ p+1 }}|z|^3 \lesssim Q^{- \frac{ p-1 }{ p+1 }} Q^{p+1}  |z|^{p+1} + Q^{\frac{ p- 1 }{ p+1 }} Q^2 z^2 \lesssim Q^{p+1}  |z|^{p+1} + t^{-2} Q^2 z^2 .
\end{equation*} 
Therefore
\begin{equation*} 
\int Q^{\frac{ 4p }{ p+1 }}|z|^3 \lesssim \mathcal N^{p+1}+t^{-2(1-\sigma)}\mathcal N^2
 \lesssim \mathcal N^{p+1}+t^{- \frac {1} {2} }\mathcal N^2.
\end{equation*} 
Since $Q^{ 2 + \frac{p-1}{p+1} } z^2 \lesssim t^{-2} Q^2 z^2 $, we have proved
\begin{equation} \label{fIntrFn1} 
 |I_4| \lesssim \int  | \partial _t Q| \GCT1 \lesssim \int Q^{\frac{ 3 p +1 }{2 (p+1) }}  \GCT1  \lesssim \mathcal N^{p+1}+t^{- \frac {1} {2} }\mathcal N^2.
\end{equation} 

We proceed similarly for $I_5$. Indeed, setting
\begin{equation*} 
\begin{split} 
\GDT2 = &  \left|f_n(U_J+Q^{\frac 12}z)-f_n(U_J)-f_n'(U_0)Q^{\frac 12}z\right|Q^{\frac 12}|z| \\ \le 
 &   \left|f_n(U_J+Q^{\frac 12}z)-f_n(U_J)-f_n'(U_J )Q^{\frac 12}z\right|Q^{\frac 12}|z|   + 
  |f_n'(U_0)-f_n'(U_J)|Qz^2
\end{split} 
\end{equation*} 
we deduce from~\eqref{taylor1} and Taylor's inequality that, with the notation~\eqref{fDefLam2}, 
\begin{equation*} 
\GDT2 \lesssim Q^{\frac {p+1}2} |z|^{p+1} + U_J^{p-\bar p} Q^{\frac {\bar p+1}2}|z|^{\bar p+1}+ U_0^{p-2}|U_0-U_J|Q|z|^2 \lesssim \GCT1.
\end{equation*} 
Using the last two inequalities in~\eqref{fIntrFn1}, we conclude that $|I_5|\lesssim \mathcal N^{p+1}+t^{-\frac 12}\mathcal N^2$.

Now, we estimate $I_6$, and we set
\begin{equation*} 
\GDT3=   | 2 f_n(U_J+Q^{\frac 12}z)-2f_n(U_J)-2f_n'(U_J)Q^{\frac 12} z-f_n''(U_J)Q z^2 | .
\end{equation*} 
By the triangle inequality, Taylor's inequality \eqref{taylor}, and $U_J^{-1}\lesssim U_0^{-1}$ (see~\eqref{e:31}),
\begin{equation*} 
\begin{split} 
\GDT3 \lesssim & | 2 f_n(U_J+Q^{\frac 12}z)-2f_n(U_J)-2f_n'(U_J)Q^{\frac 12} z-f_n''(U_0)Q z^2 |
\\ & +|f_n''(U_0)-f_n''(U_J)| Q z^2 \\ 
\lesssim & U_J^{-1} Q^{\frac{p+1}2}|z|^{p+1}+U_J^{p-\bar p -1}Q^{\frac{\bar p+1}{2}}|z|^{\bar p+1}
+U_0^{p-3}|U_0-U_J|Qz^2 \\ 
\lesssim & U_0^{-1} [ Q^{\frac{p+1}2}|z|^{p+1}+U_J^{p-\bar p }Q^{\frac{\bar p+1}{2}}|z|^{\bar p+1}
+U_0^{p-2}|U_0-U_J|Qz^2 ]
 \\ 
\lesssim & U_0^{-1} \GCT1 
\end{split} 
\end{equation*} 
with the notation~\eqref{fDefLam2}. 
Using $ |\partial_t U_0 | \lesssim U_0^{ \frac {p+1} {2} }$ and $U_0\lesssim Q^{\frac {1} {p+1}}$, 
we see that $Q  |\partial _t U_0 | \lesssim Q U_0^{ \frac {p+1} {2} } \lesssim Q U_0^{ \frac {p-1} {2} } U_0  \lesssim Q^{   \frac { 3p + 1 } {2 (p+1)}  }  U_0 $, 
hence $Q  |\partial_t U_0 | \GDT3 \lesssim Q^{   \frac { 3p + 1 } {2 (p+1)}  }  \GCT1$.
The last inequality in~\eqref{fIntrFn1} yields $|I_6|\lesssim \mathcal N^{p+1}+t^{-\frac 12}\mathcal N^2$.

Finally, we estimate $I_7$ and we set 
\begin{equation*} 
\GDT4 = |f_n(U_J+Q^{\frac 12}z)-   f_n(U_J)- f_n'(U_J)Q^{\frac 12} z| . 
\end{equation*} 
By the triangle inequality  and Taylor's expansion~\eqref{taylor}, 
\begin{equation*} 
\begin{split} 
\GDT4  \lesssim &  \Bigl|f_n(U_J+  Q^{\frac 12}z)-   f_n(U_J)- f_n'(U_J)Q^{\frac 12} z- \frac {1} {2} f''_{ n}(U_J) Q z^2 \Bigr| \\ & + \frac {1} {2}  | f'' _{ n}(U_J) |Q z^2 \\
   \lesssim & U_J^{-1} Q^{ \frac {p+1} {2}} |z|^{p+1}+ U_J ^{p-\bar p-1} Q^{ \frac {\bar p+1} {2} }  |z|^{\bar p+1} + U_J^{p-2} Q z^2
\end{split} 
\end{equation*} 
Using  $Q  | \partial _t (U_J - U_0) | \lesssim Q U_0$ (see~\eqref{e:31bis}), $U_J^{-1} \lesssim U_0^{-1}$, and $U_J \lesssim U_0$, we obtain
\begin{equation*} 
Q  | \partial _t (U_J - U_0) | \GDT4  \lesssim 
Q^{ \frac {p+3} {2}} |z|^{p+1}+ U_0 ^{p-\bar p} Q^{ \frac {\bar p+3} {2} }  |z|^{\bar p+1} + U_0^{p-1} Q^2 z^2 .
\end{equation*} 
Since $ U_0 \lesssim Q^{\frac {1} {p+1}}$ and $U_0^{p-1} \lesssim t^{-2}$, we deduce that
\begin{equation*} 
Q  | \partial _t (U_J - U_0) | \GDT4  \lesssim 
Q^{ \frac {p+3} {2}} |z|^{p+1}+ Q^{ \frac {p-\bar p} {p+1} + \frac {\bar p+3} {2} }  |z|^{\bar p+1} + t^{-2} Q^2 z^2 .
\end{equation*} 
Applying~\eqref{fEstIp1}-\eqref{festQz} for the first term and~\eqref{fEstIp2} for the second term, we see that
$|I_7|\lesssim  \mathcal N^{p+1}+t^{-\frac 12}\mathcal N^2$. 
Collecting the above estimates, we have proved \eqref{dtener}.
\medskip

\textbf{Step 3.} Conclusion.
The values of $\delta\in (0,t_J]$ and $0< \omega \le 1$ are now fixed so that \eqref{coer}, \eqref{coer2} and \eqref{dtener} hold.
Since $\mathcal N(T_n)=0$, the following is well-defined
\[
T_n^\star=\sup\{t\in [T_n,\delta]: \mbox{for all $s\in [T_n,t]$, $\mathcal N(s) \leq \omega$}\}
\]
and by continuity, $T_n^\star\in (T_n,\delta]$. For all $t\in [T_n,T^\star_n]$, using \eqref{dtener}, we find
(recall that $\lambda\in (0,\frac 12]$)
\[
\frac{d}{dt} \mathcal H \leq C \left[ t^{-1+\lambda} + t^{-\frac 12}+ 1 \right] \leq C t^{-1+\lambda}.
\]
Let $t\in [T_n,T_n^\star]$.
Since $\mathcal H(T_n)=0$, we obtain by integration on $[T_n,t]$
\[
\mathcal H(t) \leq C (t^{\lambda}-T_n^{\lambda}) \leq C (t-T_n)^{\lambda}.
\]
Therefore, using the definition of $T_n^\star$ and \eqref{coer}, for all $t\in [T_n,T_n^\star]$,
\[
 \mathcal N(t) \leq C (t-T_n)^{\frac \lambda 2}.
\]
In particular, there exists $\delta_0>0$ independent of $n$ such that, for $n$ large, it holds $T_n^\star\geq T_n+\delta_0$.
 Moreover, using \eqref{coer2},  for all $t\in [T_n,T_n+\delta_0]$, 
\[
\|(\varepsilon_n(t),\partial_t \varepsilon_n(t))\|_{H^1\times L^2}
\lesssim \mathcal N(t) \lesssim (t-T_n)^{\frac \lambda 2},
\]
which completes the proof of Proposition~\ref{pr:2}.
\end{proof}

\section{End of the proof of Theorem~\ref{TH:1}}\label{S:proof}
Let $ \Cset $ be any compact set of $\R^N$ included in the ball of center $0$ and radius~$1$
(by the scaling invariance of equation \eqref{wave}, this assumption does not restrict the generality).
It is well-known that there exists a smooth function $Z:\R^N\to [0,\infty)$ which vanishes exactly on $ \Cset $
(see \emph{e.g.}~Lemma 1.4, page 20 of \cite{MoRe}).
For $p$ as in \eqref{on:p}, choose $J$ and $k$ satisfying \eqref{def:Jk} and \eqref{on:k}. Define the function $A:\R^N\to [0,\infty)$ by
\[
A(x)=\left(Z(x) \chi( {  |x|}) + (1-  \chi( {  |x|}) ) |x|\right)^k,
\]
where $\chi $ is given by~\eqref{def:chi}.
It follows that the function $A$ satisfies \eqref{on:A} and vanishes exactly on $ \Cset $.

We consider the global solutions $u_n$ of equation \eqref{eq:un}, $\varepsilon_n$ defined by \eqref{uneps} and we set
for $0\leq t\leq t_J-T_n$,
\[
V_n(t)=U_J(T_n+t),\quad \eta_n(t)=\varepsilon_n(T_n+t),\quad \Fns_n(t)=\Ens_J(T_n+t).
\]
It follows from Proposition~\ref{pr:2} that there exist $0<\delta_0< t_J$, $0<\lambda\leq \frac 12$, and $C>0$ such that,
for $n$ large and for all $t\in [0,\delta_0]$,
\begin{equation}\label{unifeta}
\|(\eta_n(t),\partial_t \eta_n(t))\|_{H^1\times L^2}\leq C t^{\frac\lambda2}.
\end{equation}
Moreover, it follows from \eqref{eq:wn} that
\begin{equation}\label{eq:eta}
\partial_{tt}\eta_n-\Delta \eta_n = f_n(V_n+\eta_n)-f_n(V_n)+\Fns_n.
\end{equation}
Using the estimate $|f_n(u+v)-f_n(u)|\lesssim (|u|^{p-1}+|v|^{p-1})|v|$ and the embeddings
$H^1(\R^N)\hookrightarrow L^{p+1}(\R^N)$, $L^{\frac{p+1}p}(\R^N)\hookrightarrow H^{-1}(\R^N)$, we deduce that
\[
\|\partial_{tt}\eta_n\|_{H^{-1}}
\lesssim \|\eta_n\|_{H^1}+\|V_n\|_{H^1}^{p-1}\|\eta_n\|_{H^1}+\|\eta_n\|_{H^1}^{p}+\|\Fns_n\|_{L^2}
\]
so that by the estimates of Lemmas~\ref{le:0} and~\ref{le:1}, there exist $C,c>0$ such that, for all $t\in (0,\delta_0]$,
\begin{equation}\label{unifeta2}
\|\partial_{tt}\eta_n\|_{H^{-1}} \leq Ct^{-c}.
\end{equation}
Given $\tau\in (0,\delta_0)$, it follows from \eqref{unifeta} and \eqref{unifeta2} that the sequence 
$(\eta_n)_{n\geq 1}$ is bounded in $L^{\infty}((\tau,\delta_0),H^1(\R^N))\cap W^{1,\infty}((\tau,\delta_0),L^2(\R^N)) \cap W^{2,\infty}((\tau,\delta_0),H^{-1}(\R^N))$. Therefore, after possibly extracting a subsequence (still denoted by $\eta_n$), there exists $\eta\in L^{\infty}((\tau,\delta_0),H^1(\R^N))\cap W^{1,\infty}((\tau,\delta_0),L^2(\R^N)) \cap W^{2,\infty}((\tau,\delta_0),H^{-1}(\R^N))$ such that
\begin{align}
&
\eta_n \mathop{\longrightarrow}_{n\to\infty} \eta \quad \mbox{in $L^{\infty}((\tau,\delta_0),H^1(\R^N))$ weak$^*$}\label{conv1}\\
&
\partial_t\eta_n \mathop{\longrightarrow}_{n\to\infty} \partial_t\eta \quad \mbox{in $L^{\infty}((\tau,\delta_0),L^2(\R^N))$ weak$^*$}\\
&
\partial_{tt}\eta_n \mathop{\longrightarrow}_{n\to\infty} \partial_{tt}\eta \quad \mbox{in $L^{\infty}((\tau,\delta_0),H^{-1}(\R^N))$ weak$^*$}\label{conv3}\\
&
\eta_n \mathop{\longrightarrow}_{n\to\infty} \eta \quad \mbox{weakly in $H^1(\R^N)$, for all 
$t\in [\tau,\delta_0]$}\label{conv4}\\
&
\partial_t\eta_n \mathop{\longrightarrow}_{n\to\infty}\partial_t\eta \quad \mbox{weakly in $L^2(\R^N)$, for all 
$t\in [\tau,\delta_0]$}.\label{conv5}
\end{align}
Since $\tau\in (0,\delta_0)$ is arbitrary, a standard argument of diagonal extraction shows that there exists a function 
$\eta\in L^{\infty}_{\rm loc}((0,\delta_0),H^1(\R^N))\cap W^{1,\infty}_{\rm loc}((0,\delta_0),L^2(\R^N)) \cap W^{2,\infty}_{\rm loc}((0,\delta_0),H^{-1}(\R^N))$ such that (after extraction of a subsequence) \eqref{conv1}--\eqref{conv5}
hold for all $0<\tau<\delta_0$. Moreover, \eqref{unifeta} and \eqref{conv4}--\eqref{conv5} imply that
\begin{equation}\label{unifetalim}
\|(\eta(t),\partial_t\eta(t))\|_{H^1\times L^2}\leq C t^{{\frac\lambda2}}, \quad t\in (0,\delta_0),
\end{equation}
and \eqref{unifeta2} and \eqref{conv3} imply that
\begin{equation}\label{unifetalim2}
\|\partial_{tt}\eta\|_{L^\infty((\tau,\delta_0),H^{-1})}\leq C \tau^{-c}, \quad \tau\in (0,\delta_0).
\end{equation}
In addition, it follows easily from \eqref{eq:eta}, \eqref{def:Bn}, \eqref{def:fn} and the convergence properties \eqref{conv1}--\eqref{conv5} that
\begin{equation}\label{eq:w}
\partial_{tt} \eta-\Delta \eta=f(U_J+\eta)-f(U_J) + \Ens_J
\end{equation}
in $L^{\infty}_{\rm loc}((0,\delta_0),H^{-1}(\R^N))$. Therefore, setting
\[
u(t)=U_J(t)+\eta(t), \quad t\in (0,\delta_0),
\]
we observe that the function $u\in L^{\infty}_{\rm loc}((0,\delta_0),H^1(\R^N))\cap W^{1,\infty}_{\rm loc}((0,\delta_0),L^2(\R^N)) \cap W^{2,\infty}_{\rm loc}((0,\delta_0),H^{-1}(\R^N))$ and satisfies
$\partial_{tt} u-\Delta u =f(u)$ in $L^{\infty}_{\rm loc}((0,\delta_0),H^{-1}(\R^N))$.
It is  a well-known property of the energy subcritical wave equation (corresponding to assumption \eqref{on:p})
that then it holds the stronger property
\begin{equation}\label{5.12}
u\in \cont((0,\delta_0),H^1(\R^N))\cap \cont^1((0,\delta_0),L^2(\R^N))\cap\cont^2((0,\delta_0),H^{-1}(\R^N)).
\end{equation}
We refer for example to Proposition 3.1 and Lemma 2.1 in \cite{GV85}.

Finally, we prove estimates \eqref{xE} and \eqref{xnotE}.
For $x_0\not\in  \Cset $, there exist $r>0$ and $C>0$ such that $A(x)\geq C$ for all $x\in \R^N$ such that $|x-x_0|<r$. 
 In particular, there exist $0<c<C$ such that $ c \le  |U_0 (t,x)| \le C $ for $|x-x_0|<r$ and $0 <t \le \delta _0$.
Using~\eqref{e:20}, \eqref{e:22}, \eqref{e:29}, \eqref{e:31bis}  and~\eqref{e:33}, we easily deduce that      
$|U_J(x)| +  |\nabla U_J | +|\partial_t U_J(x)|\leq C'$
for some constant $C'>0$. Estimate \eqref{xnotE} then follows from~\eqref{unifetalim}.
For $x_0\in  \Cset $, \eqref{e:24}, \eqref{e:25}, \eqref{e:31} and \eqref{e:31bis} imply, for
$t\in (0,\delta_0)$,
\begin{gather*}
 t^{-\mu} \lesssim \|U_J(t)\|_{L^2(|x-x_0|<r)}\lesssim t^{-\frac {2}{p-1}}, \\
 t^{-\mu-1} \lesssim\|\partial_t U_J(t)\|_{L^2(|x-x_0|<r)}\lesssim t^{-\frac {2}{p-1}-1},
\end{gather*}
where $\mu=\frac 2{p-1}-\frac{N}{2k}$.
Estimate \eqref{xE}, and more precisely estimates \eqref{Qrp} and \eqref{dtQrp} then follow from \eqref{unifetalim}.

Now, we justify the last part of Remark~\ref{rk:2}. If $x_0\in  \Cset $ and $ \Cset $ contains a neighborhood of $x_0$ then $A(x)=0$ on this neighborhood and the lower estimate 
easily follows.
In the case where $x_0\in  \Cset $ is isolated, the function $A$ can be chosen so that $A(x)=|x|^k$ in a neighbourhood of $x_0$
(see Remark~\ref{rk:4}). In particular, by \eqref{label4} and a similar estimate for $\partial_t U_0$, we obtain
for small $r>0$, $\|u(t)\|_{L^2(|x-x_0|<r)}\lesssim t^{-\frac2{p-1}+\frac N{2k}}$ and
$\|\partial_t u(t)\|_{L^2(|x-x_0|<r)}\lesssim t^{-1-\frac2{p-1}+\frac N{2k}}$.


\begin{thebibliography} {99}

\bibitem{Alinhac}{Alinhac S.,} {{\it Blowup for nonlinear hyperbolic equations.} Progress in Nonlinear Differential Equations and their Applications, {\bf 17}, Birkh\"auser Boston, Inc., Boston, MA, 1995.}
\MScN{MR1339762} \DOI{10.1007/978-1-4612-2578-2}

\bibitem{CaFr1}{Caffarelli L. A. and Friedman A.,}
{The blow-up boundary for nonlinear wave equations, Trans. Amer. Math. Soc. \textbf{297} (1986), no. 1, 223--241.} 
\MScN{MR0849476}
\DOI{10.1090/S0002-9947-1986-0849476-3}

\bibitem{CaFr2}{Caffarelli L. A. and Friedman A.,}
{Differentiability of the blow-up curve for one dimensional nonlinear wave equations, Arch. Rational Mech. Anal. \textbf{91} (1985), no. 1 83--98.} \MScN{MR0802832}
\DOI{10.1007/BF00280224}

\bibitem{CaMaZh}{Cazenave T., Martel Y. and Zhao L.,}
{Finite-time blowup for a Schr\"o\-din\-ger equation with nonlinear source term. Discrete Contin. Dynam. Systems  {\bf 39} (2019), no. 2, 1171--1183.} 
\MScN{MR3918212} \DOI{10.3934/dcds.2019050}

\bibitem{CollotGM}{Collot C., Ghoul T.-E. and Masmoudi N.,} {Singularity formation for Burgers equation with transverse viscosity. arXiv:1803.07826}
\LINK{https://arxiv.org/abs/1803.07826}

\bibitem{CoSa}{Constantine G. M. and Savits T. H.,}
{A multivariate Faa di Bruno formula with applications, Trans. Amer. Math. Soc. \textbf{348} (1996), no.2, 503--520.}
\MScN{MR1325915}
\DOI{10.1090/S0002-9947-96-01501-2}

\bibitem{Don-Sch} {Donninger R. and Sch\"orkhuber B.,} 
{Stable self-similar blow up for energy subcritical wave equations. 
Dyn. Partial Differ. Equ.  {\bf 9}  (2012), no. 1, 63--87.}
\MScN{MR2909934} \DOI{10.4310/DPDE.2012.v9.n1.a3}

\bibitem{GV85}{Ginibre J. and Velo G.,}
{The global Cauchy problem for the non linear Klein-Gordon equation,
Math Z. \textbf{189} (1985), no. 4, 487--505.}
\MScN{MR0786279}
\DOI{10.1007/BF01168155}

\bibitem{GV89}{Ginibre J. and Velo G.,}
{The global Cauchy problem for the non linear Klein-Gordon equation-II,
Ann. Inst. H. Poincar\'e Anal. Non Lin\'eaire  {\bf 6}  (1989), no. 1, 15--35.}
\MScN{MR0984146}
\LINK{http://www.numdam.org/item/AIHPC_1989__6_1_15_0}

\bibitem{GuoHJ}{Guo Y., Had\v zi\'c M. and Jang J.} {Continued gravitational collapse for Newtonian stars. arXiv:1811.01616 [math.AP].}
\LINK{http://arxiv.org/abs/1811.01616}

\bibitem{Keller}{Keller J. B.,}
{On solutions of nonlinear wave equations, Comm. Pure Appl. Math. \textbf{10} (1957) 523--530.}
\MScN{MR0096889}
\DOI{10.1002/cpa.3160100404}

\bibitem{Kiche1}{Kichenassamy S.,}
{The blow-up problem for exponential nonlinearities, Comm. Partial Differential Equations \textbf{21} (1996), no. 1-2, 125--162.}
\MScN{MR1373767}
\DOI{10.1080/03605309608821177}

\bibitem{Kiche2}{Kichenassamy S.,}
{Fuchsian Equations in Sobolev Spaces and Blow-Up, J. Differential Equations \textbf{125} (1996), no. 1, 299--327.}
\MScN{MR1376069}
\DOI{10.1006/jdeq.1996.0033}

\bibitem{Kibook}{Kichenassamy S.,}
{ {\it Fuchsian reduction. Applications to geometry, cosmology, and mathematical physics.}
Progress in Nonlinear Differential Equations and their Applications,  {\bf 71}. Birkh\"auser Boston, Inc., Boston, MA, 2007.}
\MScN{MR2341108}
\DOI{10.1007/978-0-8176-4637-0}

\bibitem{Kicontrol}{Kichenassamy S.,}
{Control of blow-up singularities for nonlinear wave equations, Evol. Equ. Control Theory \textbf{2} (2013), no. 4, 669--677.}
\MScN{MR3177248}
\DOI{10.3934/eect.2013.2.669}

\bibitem{KiLi1}{Kichenassamy S. and Littman W.,}
{Blow-up surfaces for nonlinear wave equations. I, Comm. Partial Differential Equations \textbf{18} (1993), no. 3-4, 431--452.}
\MScN{MR1214867}
\DOI{10.1080/03605309308820936}

\bibitem{KiLi2}{Kichenassamy S. and Littman W.,}
{Blow-up surfaces for nonlinear wave equations.~II, Comm. Partial Differential Equations \textbf{18} (1993), no. 11, 1869--1899.}
\MScN{MR1243529}
\DOI{10.1080/03605309308820997}

\bibitem{KiQi}{Killip R. and Visan M.,}
{Smooth solutions to the nonlinear wave equation can blow up on Cantor sets. arXiv:1103.5257}
\LINK{https://arxiv.org/abs/1103.5257}

\bibitem{Levine}{Levine H. A.,}
{Instability and nonexistence of global solutions to nonlinear wave equations of the form
$Pu_{tt}=-Au+\mathcal{F}(u)$. Trans. Amer. Math. Soc \textbf{192} (1974), 1--21}
\MScN{MR0344697}
\DOI{10.2307/1996814}

\bibitem{Martel}{Martel Y.,} {Asymptotic $N$-soliton-like solutions of the
subcritical and critical generalized Korteweg-de Vries equations, Amer. J.
Math. {\bf 127} (2005), no. 5, 1103--1140.}
\MScN{MR2170139} \DOI{10.1353/ajm.2005.0033}

\bibitem{Merle1}{Merle F.,} {Construction of solutions with
exactly k blow-up points for the Schr\"o\-din\-ger equation with critical
nonlinearity, Comm. Math. Phys. {\bf 129} (1990), no. 2, 223--240.}
\MScN{MR1048692} \LINK{https://projecteuclid.org/euclid.cmp/1104180743}

\bibitem{MerleZ2}{Merle F. and Zaag H.,} {Isolatedness of characteristic points at blow-up for a semilinear wave equation in one space dimension. S\'eminaire: \'Equations aux Dérivées Partielles. 2009--2010, Exp. No. XI, 10 pp., S\'emin. \'Equ. D\'eriv. Partielles, \'Ecole Polytech., Palaiseau, 2012.}
\MScN{MR3098623}\LINK{http://sedp.cedram.org/sedp-bin/fitem?id=SEDP_2009-2010____A11_0}

\bibitem{MerleZ3}{Merle F. and Zaag H.,} {On the stability of the notion of non-characteristic point and blow-up profile for semilinear wave equations, Comm. Math. Phys. {\bf 333} (2015), no. 3, 1529--1562.}
\MScN{MR3302641} \DOI{10.1007/s00220-014-2132-8}

\bibitem{MoRe}{Moerdijk I. and Reyes G.,} {\emph{Models for smooth infinitesimal analysis.} Springer-Verlag, New York, 1991.}
\MScN{MR1083355} \DOI{10.1007/978-1-4757-4143-8}

\bibitem{RaphaelSzeftel}{Rapha\"el P. and Szeftel J.,} {Existence and uniqueness of minimal blow-up solutions to an inhomogeneous mass critical NLS,
J. Amer. Math. Soc. {\bf 24} (2011), no. 2, 471--546.}
\MScN{MR2748399}
\DOI{10.1090/S0894-0347-2010-00688-1}

\bibitem{Segal}{Segal I.,}
{Non-linear semi-groups. Ann. of Math. (2) \textbf{78} (1963) 339--364.}
\MScN{MR0152908}
\DOI{10.2307/1970347}

\bibitem{Speck}{Speck J.,} {Stable ODE-type blowup for some quasilinear wave equations with derivative-quadratic
nonlinearity. arXiv:1709.04778 [math.AP]}
\LINK{https://arxiv.org/abs/1709.04778}

\end{thebibliography}
\end{document}